\documentclass{article}
\usepackage[utf8]{inputenc}
\usepackage[T1]{fontenc}

\usepackage[top=1in, bottom=1in]{geometry}
\usepackage{amsmath}
\usepackage[mathscr]{eucal}

\RequirePackage[colorlinks,citecolor=blue,urlcolor=blue]{hyperref}
\usepackage{enumitem}
\usepackage{ifthen}
\usepackage{cases}
\usepackage{amsfonts,amssymb,bbm}

\newcommand{\R}{\mathbb R}

\newcommand{\positR}{(0,\infty)}

\newcommand{\C}{\mathbb C}
\newcommand{\N}{\mathbb N}
\newcommand{\E}{\mathbb E}

\newcommand{\im}{\mathbf{i}}

\newcommand{\etal}{\textit{et al}. }

\newcommand{\K}{\mathcal K}

\renewcommand{\P}{\mathbb P}

\newcommand{\bz}{\mathbf{z}}
\newcommand{\by}{\mathbf{y}}

\DeclareMathOperator*{\spec}{Spec}

\newcommand{\Yc}{\mathcal Y}

\newcommand{\ind}{\mathbbm{1}}

\newcommand{\tlambda}{\tilde{\lambda}}

\newcommand{\gNr}{\mathcal N}

\newcommand{\dd}{\,\mathrm d}
\newcommand{\tq}{\,:\,}
\newcommand{\as}{\mathrm{a.s}}
\DeclareMathOperator*{\tr}{tr}
\newcommand{\range}{\mathcal{R}}

\DeclareMathOperator*{\supp}{supp}
\DeclareMathOperator*{\diag}{diag}
\DeclareMathOperator*{\var}{Var}

\newcommand{\tran}{\top}

\newcommand{\cvweak}[1][]{%
\ifthenelse{\equal{#1}{}}{\xrightarrow{\mathcal D}}{\xrightarrow[#1]{\mathcal D}}%
}

\newcommand{\cvprob}{\xrightarrow{\mathcal P}}
\newcommand{\dlp}{d_{\mathcal{LP}}}
\usepackage{amsthm}
\newtheorem{Th}{Theorem}[section]
\newtheorem{lemma}[Th]{Lemma}
\newtheorem{prop}[Th]{Proposition}
\newtheorem{corol}[Th]{Corollary}

\theoremstyle{definition}
\newtheorem{defi}{Definition}
\newtheorem{defi-prop}[Th]{Definition-Proposition}

\newtheorem*{example*}{Example}
\theoremstyle{remark}
\newtheorem*{Rq*}{Remark}
\newtheorem{Rq}{Remark}[section]
\numberwithin{equation}{section}

\title{Joint CLT for top eigenvalues of sample covariance matrices of separable high dimensional long memory processes}
\author{Peng TIAN \footnote{Department of Statistics and Actuarial Science, The University of Hong Kong. Email: tianpeng83@gmail.com}}
\date{\today}

\begin{document}
\maketitle

\begin{abstract}
For $N,n\in\N$, consider the sample covariance matrix 
$$S_N(T)=\frac{1}{N}XX^*$$
from a data set $X=C_N^{1/2}ZT_n^{1/2}$, where $Z=(Z_{i,j})$ is a $N\times n$ matrix having i.i.d. entries with mean zero and variance one, and $C_N, T_n$ are deterministic positive semi-definite Hermitian matrices of dimension $N$ and $n$, respectively. We assume that $(C_N)_N$ is bounded in spectral norm, and $T_n$ is a Toeplitz matrix with its largest eigenvalues diverging to infinity. The matrix $X$ can be viewed as a data set of an $N$-dimensional long memory stationary process having separable dependence structure.

As $N,n\to \infty$ and $Nn^{-1} \to r\in (0,\infty)$, we establish the asymptotics and the joint CLT for  $(\lambda_1(S_N(T)),\cdots, \lambda_m(S_N(T))^\tran$ where $\lambda_j(S_N(T))$ denotes the $j$th largest eigenvalue of $S_N(T)$, and $m$ is a fixed integer. For the CLT, we first study the case where the entries of $Z$ are Gaussian, and then we generalize the result to some more generic cases. This result substantially extends our previous result in \cite{merlevede2019unbounded}, where we studied $\lambda_1(S_N(T))$ in the case where $m=1$ and $X=ZT_n^{1/2}$ with $Z$ having Gaussian entries.

In order to establish this CLT, we are led to study the first order asymptotics of the largest eigenvalues and the associated eigenvectors of some deterministic Toeplitz matrices. We are specially interested in the autocovariance matrices of long memory stationary processes. We prove multiple spectral gap properties for the largest eigenvalues and a delocalization property for their associated eigenvectors.
\end{abstract}

\section{Introduction}
\label{sec:intr}
\paragraph{Background and related work.} 
For $N,n\in\N$, we set $X=C_N^{1/2}ZT_n^{1/2}$, where $Z=(Z_{i,j})$ is a $N\times n$ matrix having i.i.d. entries with mean zero and variance one,  $C_N$ is a $N\times N$ deterministic positive semi-definite Hermitian matrix, and $T_n=(\gamma(i-j))_{1\le i,j\le n}$ be a positive semi-definite Toeplitz matrix. Then $X$ models a sample data set of an $N$-dimensional stationary process with a separable dependence structure. If the entries of $T_n$ satisfies the following power decay condition:
\begin{equation}
    \gamma(h)=\frac{L_1(|h|)}{(1+|h|)^\rho}, \label{eq:LM_gamma}
\end{equation}
or the spectral density $\varphi$ of $T_n$ has a power singularity at $0$:
\begin{equation}
    \varphi(x)=\frac{L_2(|x|^{-1})}{|x|^{1-\rho}} , \quad \text{for } x\in [-\pi,\pi], \label{eq:LM_f}
\end{equation}
where $\rho\in(0,1)$ and $L_1, L_2$ are positive, locally bounded functions which are slowly varying at $\infty$, then the process is long range dependent (LRD) or has long memory (LM) (see for example \cite{pipiras2017long}). Note that the conditions \eqref{eq:LM_gamma} and \eqref{eq:LM_f} are not equivalent but tightly related, see Section~2.2.5 of \cite{pipiras2017long}. 

From the sample data matrix $X$, we construct the sample covariance matrix 
\begin{equation}
    S:=\frac{1}{N}XX^* = \frac{1}{N}C_N^{1/2}Z T_n Z^*C_N^{1/2}\,. \label{eq:sample_cov_T}
\end{equation}
Let $m\ge 1$ be a fixed integer. We will study the asymptotics and joint CLT of the largest eigenvalues of $S$ as $N,n\to\infty$ and $r_n:=Nn^{-1}\to r\in\positR$.

For convenience of further discussion, we define the notation $S_N(\Gamma)$ by
\begin{equation}
    S_N(\Gamma):=\frac{1}{N}C_N^{1/2}Z\Gamma_n Z^*C_N^{1/2}\,. \label{eq:themodel}
\end{equation}
where $\Gamma_n$ is a $n\times n$ deterministic positive semi-definite Hermitian matrix. Then $S$ in \eqref{eq:sample_cov_T} will be denoted as $S_N(T)$, and if $\Gamma_n=I_n$ is identity, $S_N(I)$ is the classical sample covariance matrix. 

The classical sample covariance model $S_N(I)$ has been intensively studied in the last decades. These studies are mostly concentrated on the global behaviors of the spectrum, including limiting spectral distribution (LSD) (\cite{marvcenko1967distribution, wachter1978strong,jonsson1982some,yin1986limiting,silverstein1995empirical,silverstein1995strong}) and CLT for linear spectral statistics (\cite{bai2004clt,bai2010functional,guedon2014central,najim2016gaussian}); and also the local behaviors of individual eigenvalues (\cite{baik2005phase,johansson2000shape,johnstone2001distribution,karoui2007tracy,lee2016tracy,bao2015universality,knowles2017anisotropic,bai2008central,bai2012sample}). 

Recently several models of matrices $S_N(I)$ with $C_N$ having a small number of divergent eigenvalues have been considered, in the context of principal component analysis (PCA) \cite{jung2009pca, shen2013surprising, wang2017asymptotics, cai2017limiting} and long memory processes \cite{merlevede2019unbounded}. Although the assumptions in these various works differ, many results coincide with the degenerated case of Bai and Yao \cite{bai2012sample} after normalization (see  for example \cite{merlevede2019unbounded}). %Remark~\ref{rq:assumpt_comparation}).

The model $S_N(I)$ assumes that the columns of the data matrix $X$ are i.i.d. However this is not always the case in the practical applications. The separable model $S_N(\Gamma)$ introduces a special type of correlations between columns, or different weights on columns, achieving certain balance between generality and simplicity. So it attracts more and more attention nowadays. A first result on this model is due to Zhang \cite{lixin2006spectral} on the LSD of $S_N(\Gamma)$. She proved that if the empirical spectral distribution (ESD) $\mu^{\Gamma_n}$ of $\Gamma_n$ and the ESD  $\mu^{C_N}$ of $C_N$ converge weakly to $\nu_\Gamma$ and $\nu_C$  respectively, then as $N,n\to\infty$,  the ESD $\mu^{S_N(\Gamma)}$ of $S_N(\Gamma)$ will converge weakly to a non-random probability measure $\mu$ for which if $\nu_C=\delta_0$ or $\nu_\Gamma=\delta_0$, then $\mu=\delta_0$; otherwise for each $z\in\C^+:=\{z\in\C\tq \Im z>0\}$, the Cauchy-Stieltjes transform $s(z)$ of $\mu$, together with another two functions, denoted by $g_\Gamma(z)$ and $g_C(z)$, $(s(z),g_\Gamma(z),g_C(z))$ is the unique solution in the set
$$U=\left\{(s(z),g_\Gamma(z),g_C(z))\tq \Im s(z)<0, \quad \Im(zg_\Gamma(z))<0,\quad \Im(zg_C(z))<0\right\}$$
to the following system of equations
\begin{equation}
    \begin{cases} s(z)=z^{-1}(1-r)+z^{-1}r\int\frac{1}{1-g_C(z)x}\dd \nu_\Gamma(x)\,, \\
    s(z)=z^{-1}\int\frac{1}{1-g_\Gamma(z)x}\dd\nu_C(x)\,, \\
    s(z)=z^{-1}+g_\Gamma(z)g_C(z)\,.
    \end{cases} \label{eq:sys_equation}
\end{equation}
Later, Paul and Silverstein \cite{paul2009no} proved in the case where $\Gamma_n$ is diagonal, that almost surely, for large enough $N,n$, there is no eigenvalue of $S_N(\Gamma)$ in any closed interval outside the support of the limiting spectral distribution (LSD). This is an extension of the results of \cite{bai1998no} for $S_N(I)$. Couillet and Hachem \cite{couillet2014analysis} studied the analytical properties of the LSD $\mu$ when $\nu_\Gamma\ne\delta_0$ and $\nu_C\ne\delta_0$, including the determination of the support of $\mu$, extending the work of Choi and Silverstein \cite{silverstein1995analysis} for $S_N(I)$ to the separable model $S_N(\Gamma)$. The CLT for linear spectral statistics has also been studied by Bai \etal \cite{zhidong2016central} and Li \etal \cite{li2019central}.

Regarding the extreme eigenvalues of $S_N(\Gamma)$, far less is known compared to the classical model $S_N(I)$. In \cite{yang2018edge}, Yang proved the edge universality under the condition that the densities of the LSD's $\nu_\Gamma,\nu_C$ have a regular square-root behavior at the rightmost edge (soft edge). With this result, if we find the fluctuations of the largest eigenvalue $\lambda_1(S_N(\Gamma))$ at a soft edge in the case where the entries $Z_{i,j}$ are Gaussian, then the fluctuations at a soft edge in general cases will be determined. However, even in the Gaussian case, these fluctuations are still unknown.

The general spiked eigenvalues of $S_N(\Gamma)$ are also a new topic in the recent studies. In \cite{zhang2018clt}, a very particular case of this problem has been touched. One can refer to Remark~\ref{rq:order_theta_j_m1} for the relations between the concerned results of \cite{zhang2018clt} and the ours. During the preparation of the present paper, we learned that a newly submitted paper \cite{ding2019spiked} treated the general spike separable model $S_N(\Gamma)$ with $C_N, \Gamma_n$ general Hermitian matrices having a finite number of spikes. The authors studied the asymptotics and large deviations (instead of joint CLT) of spiked eigenvalues of $S_N(\Gamma)$, and the associated eigenvectors. The main restriction of \cite{ding2019spiked} is that they assumed that $C_N$ and $\Gamma_n$ are all bounded in spectral norm, and that the spectrums of both $C_N$ and $\Gamma_n$ do not concentrate at zero, also that the number of spikes is finite. These assumptions exclude our model $S_N(T)$ from applying their results.

\paragraph{Introduction to the results.} The present paper aims at studying the asymptotics and the joint CLT of $m$ (with $m\ge 1$ an arbitrary fixed integer) largest eigenvalues of $S_N(T)$. The basic idea is analogous to the previous article \cite{merlevede2019unbounded} but we extend substantially the results of that paper. 

In order to study the largest eigenvalues of $S_N(T)$, the asymptotics of largest eigenvalues and the associated eigenvectors of $T_n$ satisfying \eqref{eq:LM_gamma} or \eqref{eq:LM_f} will be studied. In \cite{merlevede2019unbounded} we proved that the ESD of $T_n$ converges weakly to a non-compact supported measure, thus the largest eigenvalues of $T_n$ diverges to infinity. We also studied the asymptotic behavior of largest eigenvalues of $T_n$ if it satisfies \eqref{eq:LM_gamma}, and proved that $\lambda_j(T_n)\sim n^{1-\rho}L_1(n)\lambda_j(\K^{(\rho)})$ as $n\to\infty$ and $j\ge 1$ is fixed, where $\K^{(\rho)}$ is a compact operator defined in \eqref{eq:defi_K_rho} below. By proving the simplicity of $\lambda_1(\K^{(\rho)})$, we proved also the spectral gap property for the largest eigenvalue of $T_n$, that is, $\lambda_2(T_n)/\lambda_1(T_n)$ is bounded by a constant smaller than $1$. However, it is well known that the two conditions \eqref{eq:LM_gamma} and \eqref{eq:LM_f} are not equivalent without the quasi-monotone conditions on $L_1$ or $L_2$. For example, Gubner  \cite{gubner2005theorems} gave counterexamples in both directions of implication. In this paper, we will prove that if $T_n$ satisfies \eqref{eq:LM_f}, without assuming quasi-monotonicity of $L_2$, we have
\begin{equation}
    \lambda_j(T_n)\sim 2\Gamma(\rho)\sin\left(\frac{\rho\pi}{2}\right)n^{1-\rho}L_2(n)\lambda_j(\K^{(\rho)})
\end{equation}
as $n\to\infty$, where $\Gamma$ is the Gamma-function. We will also prove that all nonzero eigenvalues of $\K^{(\rho)}$ are simple, using a method totally different from \cite{merlevede2019unbounded}. In consequence, the spectral gap property holds for any finite number of largest eigenvalues of $T_n$. Furthermore, we will study the relation between the eigenvectors of $T_n$ and the eigenfunctions of $\K^{(\rho)}$, and prove that the eigenvectors associated with the largest eigenvalues of $T_n$ are delocalized. These results may have independent interests. 

Using the results on Toeplitz matrices, we study the asymptotics and fluctuations of largest eigenvalues of $S_N(T)$ as $N,n\to\infty$ with $N/n\to r\in\positR$. We first prove the following asymptotic behavior of $\lambda_j(S_N(T))$ for any fixed $j$. 
Assume that the entries of $Z$ have finite fourth moment, and some other conditions, then for any $j\ge 1$, in probability,
\begin{equation}
    \lambda_j(S_N(T))\sim \lambda_j(T_n)\frac{\tr C_N}{N}. \label{eq:intro:equiv}
\end{equation}
Moreover, if the entries $Z_{i,j}$'s are Gaussian, \eqref{eq:intro:equiv} holds almost surely. 

Then we will build the joint CLT of largest eigenvalues for the generic model $S_N(\Gamma)$ which implies in particular the following results: suppose that $T_n$ satisfies \ref{eq:LM_gamma} or \ref{eq:LM_f}, and $Z$ has i.i.d. standard Gaussian entries, with $C_N$ satisfies some mild conditions, then as $N,n\to\infty$ with $N/n\to r\in \positR$,
\begin{equation}
    \sqrt{N}\begin{pmatrix}\frac{\lambda_1(S_N(T))}{\lambda_1(T_n)}-\theta_1 & \cdots &  \frac{\lambda_m(S_N(T))}{\lambda_m(T_n)}-\theta_m \end{pmatrix}^\tran \cvweak \gNr(0,\sigma^2 I_m) \label{eq:intro_result_clt1}
\end{equation}
where $\sigma^2=2$ when the entries of $Z$ are real Gaussian, and $\sigma^2=1$ when the entries of $Z$ are complex Gaussian, and $\theta_j$ are determined by some equations. We then generalize the result to the non-Gaussian case.

If the entries of $Z$ are not Gaussian, it is more complicated. The CLT of eigenvalues of $S_N(T)$ depends also on the eigenvectors of $T_n$ and $C_N$. We will prove a general theorem implying that, when $C_N$ is diagonal, and the parameter $\rho$ of $T_n$ defined in \eqref{eq:LM_gamma} or \eqref{eq:LM_f} belongs to $(0,3/4)$, that is, the decay of the correlation of the process is sufficiently slow, then \eqref{eq:intro_result_clt1} still holds with the same $\theta_j$ and $\sigma^2=1+|\E Z_{1,1}^2|^2$, where $Z_{1,1}$ is the entry of $Z$ in the first row and first column. Note that if $Z_{1,1}$ is real with variance one, or if $Z_{1,1}$ is complex with $\E Z_{1,1}^2=0$, we will get the same CLT as the Gaussian case. This phenomenon is due to the delocalization of the largest eigenvectors of Toeplitz matrices $T_n$. 

\paragraph{Organizations.} This paper is organized as follows. In Section~\ref{sec:main_results} we state our main theorems. This section is divided in three parts. In \ref{subsec:toeplitz}, we state the results on Toeplitz matrices $T_n$; in \ref{subsec:convergence_eigenvalue}, we state the asymptotics of the largest eigenvalues of $S_N(\Gamma)$; in \ref{subsec:CLT_diag} we state the CLT for largest eigenvalues of $S_N(\Gamma)$ in the case where $C_N, \Gamma_n$ are diagonal, or where $Z_{i,j}$ are Gaussian; in \ref{subsec:generalclt}, we present some generalizations of the CLT with non-diagonal $\Gamma_n$. The other sections contains the proofs of these results.

\paragraph{Notations.} For a Hermitian operator or matrix $A$, we denote its real eigenvalues by decreasing order as
$$\lambda_1(A) \ge \lambda_2(A)\ge \dots $$
For a matrix or a vector $A$, we use $A^\tran$ to denote the transpose of $A$, and $A^*$ the conjugate transpose of $A$. For a $N\times N$ matrix $A$, we denote the ESD of $A$ by $\mu^A$, which is defined by $\mu^A:=N^{-1}\sum_{i=1}^N \delta_{\lambda_i(A)}$, where $\delta_x$ is the Dirac measure at $x$.

The kernel of a linear operator $A: X\rightarrow X$ is denoted by $\ker A$.
The spectrum of $A$ is denoted by $\spec(A)$. We denote the $L^p$ or $l^p$ norm by $\|\cdot\|_p$. For a matrix or a linear operator $A$, the operator norm of $A$ induced by vector norm $\|\cdot\|_p$ is denoted by $\|A\|_p$, and we recall that $\|A\|_p:=\sup_{\|v\|_p=1}\|Av\|_p$. The $L^2$ or $l^2$ norm will be abbreviated as $\|\cdot\|$. We say that a function $f$ or a vector $v$ is "normalized" or "unit length" when $\|f\|=1$ or $\|v\|=1$. When functions or vectors are said to be "orthonormal", they will be implicitly considered as elements of a Hilbert space.

For two probability measures $P$ and $Q$ on $\R^m$, we denote their L\'evy-Prokhorov distance by $\dlp(P,Q)$ which is defined by
\begin{equation}
    \dlp(P,Q):=\inf\{\varepsilon \tq P(A)\le Q(A^\varepsilon)+\varepsilon, Q(A)\le P(A^\varepsilon)+\varepsilon , \forall A\in\mathcal B(\R^m)\},
\end{equation}
where $A^\varepsilon$ is defined by
$$A^\varepsilon:=\{x\in\R^m \tq \exists y\in A, \text{ s.t. }\|x-y\|<\varepsilon\}.$$
It is well known that this distance metrizes the weak convergence. For two random variables $X,Y\in\R^m$ with distributions $\mu_X, \mu_Y$, respectively, we sometimes write $\dlp(X,Y)$, $\dlp(X,\mu_Y)$ or $\dlp(\mu_X,Y)$ which all mean $\dlp(\mu_X,\mu_Y)$.

Given $x\in\R$, we denote by $\lfloor x\rfloor$ the integer satisfying $ \lfloor x\rfloor\le x <\lfloor x\rfloor+1$. Given two sequences of non-negative numbers $x_n, y_n$, $x_n\sim y_n$ means that $\lim_{n\to\infty}x_n/y_n=1$. If $X_n,X$ are random variables, the notation $X_n=o_P(1)$ means that $\lim_{n\to\infty} X_n=0$ in probability. The notations $X_n\cvweak X$ and $X_n\cvprob X$ denote convergence in distribution and in probability, respectively. If $\mu_n,\mu$ are measures, we denote with a slight abuse of notation $\mu_n\cvweak \mu$ for the weak convergence of $\mu_n$ to $\mu$. 

In order to estimate some quantities we need sometimes split it into several parts. When we use $P_1, P_2, \dots$ to denote these parts, their definition is limited in the same proof, the same section or subsection.

\begin{defi} \label{def:proba_termino}
We say that a sequence of events $E_n$ hold {\em with high probability}, if $\P(E_n)=1-o(1)$; {\em with low probability}, if $\P(E_n)=o(1)$; {\em with overwhelming probability}, if for any $M>0$, $\P(E_n)=1-o(n^{-M})$; {\em with tiny probability}, if for any $M>0$, $\P(E_n)=o(n^{-M})$.
\end{defi}

The cardinal of a set $B$ is denoted by $\# B$. In the proofs we use $K$ to denote a constant that may take different values from one place to another. If the constant depends on some parameter $p$, we denote the constant by $K_p$.

\paragraph{Acknowledgement.} Part of this work was completed during my PhD study, and was supported financially by Universit\'e Paris-Est and Labex B\'ezout. I would like to thank gratefully my Ph-D advisers Professor Florence M\`ERLEVEDE and Professor Jamal NAJIM. Also thanks to Professor Jianfeng YAO in the University of Hong Kong for fruitful discussions.

\section{Main theorems}
\label{sec:main_results}

\subsection{Spectral properties of Toeplitz matrices}
\label{subsec:toeplitz}
We collect our results on Toeplitz matrices in this section. Let $T_n=\left(\gamma(i-j)\right)_{i,j=1}^n$ satisfy \eqref{eq:LM_gamma}. Let $\K^{(\rho)}$ be the operator defined on $L^2(0,1)$ by
\begin{equation}
(\K^{(\rho)}f)(x)=\int_0^1 \frac{f(y)}{|x-y|^\rho}\dd y\,, \quad \text{ for } f\in L^2(0,1)\,. \label{eq:defi_K_rho}
\end{equation}

In \cite{merlevede2019unbounded}, we have established the relation between the eigenvalues of $T_n$ and the eigenvalues of $\K^{(\rho)}$. From the proof of Theorem~2.3 of \cite{merlevede2019unbounded} we know that the operator $\K^{(\rho)}$ is compact and positive semi-definite. It has infinitely many positive eigenvalues. And if $T_n$ is defined by \eqref{eq:LM_gamma}, then for any $j\ge 1$, we have
\begin{equation}
    \lim_{n\to\infty}\frac{\lambda_j(T_n)}{n^{1-\rho}L_1(n)} = \lambda_j(\K^{(\rho)})>0\,. \label{eq:Toep_asy_eval}
\end{equation}
Using the min-max formula for the largest eigenvalue and an argument by absurd, we have also proved in \cite{merlevede2019unbounded} that $\lambda_1(\K^{(\rho)})$ is simple, so that we proved the spectral gap property for the largest two eigenvalues of $T_n$:
$$\lim_{n\to \infty}\frac{\lambda_2(T_n)}{\lambda_1(T_n)}<1\,.$$

In this paper, using a different method, we will prove that all non-zero eigenvalues of $\K^{(\rho)}$ are simple. As a consequence we prove the multiple spectral gap property for any $j$th largest eigenvalue of $T_n$:
\begin{equation}
    \lim_{n\to \infty}\frac{\lambda_{j+1}(T_n)}{\lambda_j(T_n)}<1\,. \label{eq:spectral_gap_j}
\end{equation}

\begin{prop}\label{prop:simplicity_eigen} All non-zero eigenvalues of the operator $\K^{(\rho)}$ defined by \eqref{eq:defi_K_rho} with $\rho\in(0,1)$ are simple, and the associated eigenfunctions are continuous in $[0,1]$.
\end{prop}

We note that $\K^{(\rho)}$ is self-adjoint, so for any non-zero eigenvalue $\lambda$, its algebraic multiplicity equals to its geometric multiplicity, which is defined as $\dim \ker (\lambda I-\K^{(\rho)})$. For more information about algebraic multiplicity, see \cite{lopez2007algebraic}. So here to say that a non-zero eigenvalue $\lambda$ is simple, means
$$\dim \ker \left(\lambda I-\K^{(\rho)}\right)=1\,.$$

In the next proposition, we provide a quantitative description of the  eigenvectors associated with $\lambda_j(T_n)$ for any fixed $j$.

\begin{prop} \label{prop:asympt_eigenvect} Let $T_n$ satisfy \eqref{eq:LM_gamma}. For any $j\ge 1$, let $f_j$ be the normalized eigenfunction of $\K^{(\rho)}$ associated with $\lambda_j(\K^{(\rho)})$, and $u_j=(u_{j,1},\dots,u_{j,n})^\tran$ be a normalized eigenvector of $T_n$ associated with $\lambda_j(T_n)$. Then, up to a change of sign, we have
\begin{equation}\lim_{n\to\infty}\sup_{1\le k\le n}\left\{\left|\sqrt{n}u_{j,k}-f_j\left(\frac{k}{n}\right)\right|\right\} =0 \,. \label{eq:th:asympt_eig_vec}
\end{equation}
\end{prop}

From this proposition we deduce the delocalization of eigenvector $u_j$ associated with $\lambda_j(T_n)$ for any fixed $j\ge 1$. Indeed by \eqref{eq:th:asympt_eig_vec}, for large enough $n$, we have
$$\|u_j\|_\infty\le \frac{1+\|f_j\|_\infty}{\sqrt{n}}\,,$$
and because $f_j$ is continuous on $[0,1]$, we have $\|f_j\|_\infty<\infty$. Thus we conclude that 
$$\|u_j\|_\infty\to 0\,.$$

The above propositions also applies to some Toeplitz matrices satisfying \eqref{eq:LM_f}. It is well known that if $T_n$ satisfies \eqref{eq:LM_f} with quasi-monotonic $L_2$ (see Section~2.2.5 of \cite{pipiras2017long} for definition), then $T_n$ also satisfies \eqref{eq:LM_gamma} with
\begin{equation}
    L_1(h)\sim L_2(h)2\Gamma(\rho)\sin\left(\frac{\rho\pi}{2}\right)\quad \text{as } h\to\infty\,, \label{eq:relation_L1L2}
\end{equation}
where $\Gamma$ is the Gamma-function. In particular, if $L_2\equiv 1$, then $L_1(h)$ tends to a constant as $h\to\infty$, and \eqref{eq:Toep_asy_eval}, Proposition~\ref{prop:asympt_eigenvect} hold. 

Without the condition of quasi-monotonicity on $L_1$ or $L_2$, the conditions \eqref{eq:LM_gamma} and \eqref{eq:LM_f} are not equivalent. See \cite{gubner2005theorems} for counterexamples in both directions. However, thanks to the following theorem, the above results can be extended to $T_n$ defined by \eqref{eq:LM_f} with general slowly varying function $L_2$. 

\begin{Th}\label{th:T_Tprime} Let $T_n$ and $T'_n$ be $n\times n$ Toeplitz matrices both satisfying \eqref{eq:LM_f} with the same $\rho\in(0,1)$, with an arbitrary slowly varying function $L_2$ for $T_n$, and with $L_2'\equiv 1$ for $T'_n$. Then  
\begin{equation}
    \left\|\frac{T_n}{n^{1-\rho}L_2(n)}-\frac{T'_n}{n^{1-\rho}}\right\|\xrightarrow[n\to\infty]{} 0. \label{eq:T_Tprime_norm}
\end{equation}
In consequence, for any fixed $j\ge 1$, we have
\begin{equation}
    \lim_{n\to\infty}\frac{\lambda_j(T_n)}{n^{1-\rho}L_2(n)} = 2\Gamma(\rho)\sin\left(\frac{\rho\pi}{2}\right)\lambda_j(\K^{(\rho)})>0\,. \label{eq:Toep_f_asy_eval}
\end{equation}
Moreover, if $u_j$ (resp. $u_j'$) is the eigenvector of $T_n$ (resp. $T'_n$) associated with the $j$th largest eigenvalue, then
\begin{equation}
    \|u_j-u_j'\|_\infty\le \|u_j-u_j'\|\xrightarrow[n\to\infty]{} 0\,. \label{eq:T_Tprime_evect}
\end{equation}
\end{Th}

Let $\tilde T_n:=T_n/\|T_n\|$. The following proposition provides the decay of moments of the ESD $\mu^{\tilde T_n}$ for $T_n$ satisfying \eqref{eq:LM_gamma} or \eqref{eq:LM_f}. This result shows that the $T_n$ with parameter $\rho\in(0,3/4)$ satisfies \ref{ass:contr_trace_1} or \ref{ass:contr_trace_2} below, and is needed in the proof of Theorem~\ref{th:joint_fluct}. 

\begin{prop}\label{prop:decay_moments} Let $T_n$ be $n\times n$ Toeplitz matrix satisfying \eqref{eq:LM_gamma} or \eqref{eq:LM_f}, and let $\tilde T_n=T_n/\|T_n\|$.
\begin{enumerate}
    \item \label{enum:th:tech_assumpt_1} If $\rho \in(0,1/2)$, then 
    \begin{equation}
        \int x\dd\mu^{\tilde T_n}(x)=\frac{\tr \tilde T_n}{n}=o(1/\sqrt{n})\,. \label{eq:enum:th:tech_assumpt_1}
    \end{equation}
    \item\label{enum:th:tech_assumpt_2} If $\rho\in[1/2,3/4)$, then
    \begin{equation}
        \int x^2\dd\mu^{\tilde T_n}(x)=\frac{\tr \tilde T_n^2}{n}=o(1/\sqrt{n})\,. \label{eq:enum:th:tech_assumpt_2}
    \end{equation}
\end{enumerate}
\end{prop}

\subsection{Convergence of largest eigenvalues of separable sample covariance matrix}
\label{subsec:convergence_eigenvalue}
For $N,n\in\N$, let
\begin{equation}
    S_N(\Gamma):=\frac{1}{N}C_N^{1/2}Z\Gamma_nZ^*C_N^{1/2}, \label{eq:def_S_N_Gamma}
\end{equation}
where $C_N, \Gamma_n$ are $N\times N$ and $n\times n$ deterministic positive semi-definite Hermitian matrices, and $Z$ is a $N\times n$ matrix having i.i.d. entries $Z_{i,j}$. Let
$$c_1\ge c_2 \ge \dots \ge c_N \quad\text{and}\quad t_1\ge t_2 \ge \dots \ge t_n$$
be the eigenvalues of $C_N$ and $\Gamma_n$ respectively. Let $m\ge 1$ be a fixed integer. We assume that the following assumptions hold:
\begin{enumerate}[leftmargin=*, label={\bf A\arabic{enumi}}]
    \item \label{ass:entries_cdts}The entries $Z_{i,j}$ satisfy
    $$\E Z_{i,j}=0, \qquad \E |Z_{i,j}|^2=1 \qquad\text{and}\quad \E|Z_{i,j}|^4<\infty\,,$$
    \item \label{ass:C_N}The spectral norm $C_N$ is bounded in $N$, and the ESD $\mu^{C_N}$ of $C_N$ converges weakly as $N\to\infty$, to a probability measure $\nu_C\ne \delta_0$. 
    \item \label{ass:eigenvalue_conv_Gamma}There exists a decreasing sequence of positive numbers
    $$a_1 \ge a_2 \ge \cdots$$
    converging to $0$ such that for any $j\ge 1$, we have
    $$\lim_{n\to\infty}t_j=a_j\,.$$
\end{enumerate}
Note that under \ref{ass:eigenvalue_conv_Gamma}, we have $\mu^{\Gamma_n}\cvweak \delta_0$, and for any $\varepsilon>0$,
$$\sup_{n}\#\{j\tq t_j>\varepsilon\}<\infty.$$

For further use, we will prove a concentration inequality for the largest eigenvalues of $S_N(\Gamma)$ with the following conditions.
\begin{enumerate}[leftmargin=*, label={\bf A\arabic{enumi}}]
\setcounter{enumi}{3}
    \item \label{ass:C_diagonal} The matrices $C_N$ are diagonal:
    $$C_N=\diag(c_1,\dots,c_N).$$
    \item \label{ass:Gamma_diagonal} The matrices $\Gamma_n$ are diagonal:
    $$\Gamma_n=\diag(t_1,\dots,t_n).$$
    \item \label{ass:bound_cdt}(Bound condition) There exists a sequence of positive numbers $\varepsilon_n\to 0$ such that almost surely for large enough $N,n$, 
    $$|Z_{i,j}|\le \varepsilon_n\sqrt{n}\,,$$
\end{enumerate}

\begin{Rq} We take two examples for which the bound condition \ref{ass:bound_cdt} holds. The first case is where $\E|Z_{i,j}|^{6+\epsilon}<\infty$ for some $\epsilon>0$. In this case, we have
$$\sum_{n=1}^\infty Nn\P(|Z_{i,j}|>\varepsilon_n\sqrt{n})\le \sum_{n=1}^\infty \frac{K}{\varepsilon_n^{6+\epsilon}n^{1+\epsilon/2}}<\infty,$$
where we have assumed that the convergence rate of $\varepsilon_n$ to $0$ is slower than any preassigned rate. Then by Borel-Cantelli's Lemma, the bound condition holds. The second case is where $\E|Z_{i,j}|^4<\infty$, and $Z_{i,j}$ does not depend on $N,n$ for any fixed $i,j$. In other words, $Z_{i,j}$ are all from an infinite double array $(Z_{i,j})_{i,j\ge 1}$. In this case, by the truncation lemma 2.2 of \cite{yin1988limit}, the bound condition holds.
\end{Rq}

Recall that we use $N,n\to\infty$ to denote $N,n\to\infty, N/n\to r\in\positR$.

\begin{prop}\label{prop:conv_eigen_SN} Let $S_N(\Gamma)$ be defined as \eqref{eq:def_S_N_Gamma}. Under \ref{ass:entries_cdts}, \ref{ass:C_N} and \ref{ass:eigenvalue_conv_Gamma}, for any $j\ge 1$, we have
$$\lambda_j(S_N(\Gamma)) \xrightarrow[N,n\to\infty]{\mathcal P} a_j\int x\dd\nu_C(x).$$
Moreover, if $Z_{i,j}$'s are real or complex Gaussian, or if \ref{ass:C_diagonal}, \ref{ass:Gamma_diagonal} and \ref{ass:bound_cdt} hold, then the above convergence holds almost surely.
\end{prop}

\begin{Rq}
The almost sure convergence under \ref{ass:C_diagonal}, \ref{ass:Gamma_diagonal} and \ref{ass:bound_cdt} is in fact a byproduct of Lemma~\ref{lemma:high_proba_concentration} which is needed in the proof of CLT \ref{th:diag_clt}. However this does not allow to conclude the a.s. convergence when $Z_{i,j}$'s are Gaussian. Indeed if the entries of $Z$ are i.i.d real Gaussian variables, and if $C_N$ or $\Gamma_n$ are complex and non-diagonal, then we cannot diagonalize $C_N$ or $\Gamma_n$ because the real Gaussian vectors are not unitary invariant. Thus we will proceed an independent proof for Gaussian case with help of a Gaussian concentration inequality.
\end{Rq}

Applying the above generic result to the special case of $S_N(T/\|T\|)$, we obtain the following result:
\begin{corol} Let $T_n$ be a sequence of Toeplitz matrices satisfying \eqref{eq:LM_gamma} or \eqref{eq:LM_f}. Let $S_N(T)$ be defined as before. Then if \ref{ass:entries_cdts}, \ref{ass:C_N} hold, for any fixed $j\ge 1$ we have
$$\frac{\lambda_j(S_N(T))}{\lambda_j(T_n)}\xrightarrow[N,n\to\infty]{\mathcal P}\int x\dd\nu_C(x).$$
Moreover, if $Z_{i,j}$'s are standard real or complex Gaussian, then the above convergence holds almost surely.
\end{corol}

\subsection{CLT for largest eigenvalues: Diagonal \& Gaussian case}
\label{subsec:CLT_diag}
In this section, we assume that $C_N$, $\Gamma_n$ are diagonal, and study the CLT for largest eigenvalues of $S_N(\Gamma)$. As a corollary, we obtain the result for Gaussian case. 

\begin{enumerate}[leftmargin=*, label={\bf A\arabic{enumi}}]
\setcounter{enumi}{6}
    \item \label{ass:moment_6} The sixth moment of the entries is finite:
    $$\E|Z_{1,1}|^6<\infty.$$
    \item \label{ass:spectral_gap} The $m$ largest eigenvalues of $\Gamma_n$ satisfy the multiple spectral gap property:
    $$\lim_{n\to\infty}t_{j+1}=a_{j+1}<a_j=\lim_{n\to\infty}t_j \quad \text{ for }j=1,\dots,m.$$
\end{enumerate}

For $x,z\in\C$ we define
\begin{equation}
    G(x,z):=\frac{1}{N}\sum_{i=1}^N\frac{c_i}{x-zc_i}\,.
\end{equation}
For $j=1,\dots,m$, let $\theta_j$ be the largest solution of the equation
\begin{equation}
    G\left(\theta_j,\frac{1}{N}\sum_{k\ne j}\frac{t_k}{t_j-t_k}\right)=\frac{1}{N}\sum_{i=1}^N\frac{c_i}{\theta_j-\left(\frac{1}{N}\sum_{k\ne j}\frac{t_k}{t_j-t_k}\right)c_i}=1\,. \label{eq:theta_equation}
\end{equation}

\begin{Rq} \label{rq:on_equation_G} Note that if not all $c_i$'s are $0$, and if $z\in\R$, then from the graph of the function $x\mapsto G(x,z)$, we see that the equation $G(x,z)=1$ on $x$ admits $\#\{zc_i\tq c_i\ne 0,i=1,\dots,n\}$ real solutions. 

Moreover, we prove that the largest solution of \eqref{eq:theta_equation} tends to $\int x\dd\nu_C(x)$. Indeed, we know that under the assumptions \ref{ass:eigenvalue_conv_Gamma} and \ref{ass:spectral_gap}, we have 
$$N^{-1}\sum_{k\ne j}\frac{t_k}{t_j-t_k}\xrightarrow[N,n\to\infty]{} 0\,,$$ and the assumption \ref{ass:C_N} ensures that
$$\frac{1}{N}\sum_{i=1}^N c_i\xrightarrow[N\to\infty]{}\int x\dd\nu_C(x)\ne 0.$$
Also note that for every fixed $\theta\ne 0$,
$$G\left(\theta,\frac{1}{N}\sum_{k\ne j}\frac{t_k}{t_j-t_k}\right)\xrightarrow[N,n\to\infty]{}\theta^{-1}\int x\dd\nu_C(x).$$
Thus for any $0<\epsilon<\int x\dd\nu_C(x)$, let $\theta^{(1)}=\int x\dd\nu_C(x)-\epsilon$, $\theta^{(2)}=\int x\dd\nu_C(x)+\epsilon$, then we can see that asymptotically the largest solution of the equation \eqref{eq:theta_equation} is between $\theta^{(1)}$ and $\theta^{(2)}$.
\end{Rq}

We define
\begin{equation}
    \Lambda_m(\Gamma) := \sqrt{N}\begin{pmatrix}\frac{\lambda_1(S_N(\Gamma))}{\lambda_1(\Gamma)} - \theta_1 & 
        \cdots & 
    \frac{\lambda_m(S_N(\Gamma))}{\lambda_m(\Gamma)} - \theta_m \end{pmatrix}^\tran. \label{eq:def_Lambda_m}
\end{equation}

\begin{Th} \label{th:diag_clt}
Under \ref{ass:entries_cdts}, \ref{ass:C_N}, \ref{ass:eigenvalue_conv_Gamma}, \ref{ass:C_diagonal}, \ref{ass:Gamma_diagonal} and \ref{ass:moment_6}, \ref{ass:spectral_gap}, we have
$$\Lambda_m(\Gamma)\cvweak[N,n\to\infty] \gNr\left(0,(\E|Z_{1,1}|^4-1)\int x^2\dd\nu_C(x) I_m\right)\,.$$
\end{Th}

For general non-diagonal $C_N$ and $\Gamma_n$, note that if $Z_{i,j}$ are standard complex Gaussian, or if $Z_{i,j}$ are standard real Gaussian and $C_N, \Gamma_n$ are both real, then the eigenvalues of $S_N(\Gamma)$ have the same joint distribution with the eigenvalues of 
$$\frac{1}{N}\diag(\sqrt{c_1},\dots,\sqrt{c_N})Z\diag(t_1,\dots,t_n)Z^*\diag(\sqrt{c_1},\dots,\sqrt{c_N})\,.$$
Therefore the CLT \ref{th:diag_clt} applies to the Gaussian case no matter whether $C_N, \Gamma_n$ are diagonal. More particularly, applying the above result to $S_N(T)$ with $T_n$ the Toeplitz matrix defined as before, we get the following corollary.
\begin{corol}
Let $T_n$ be a sequence of Toeplitz matrices satisfying \eqref{eq:LM_gamma} or \eqref{eq:LM_f}. Let $S_N(T)$ be defined as before. Let $\Lambda_m(T)$ be defined by replacing $\Gamma$ with $T$ in \eqref{eq:def_Lambda_m}. Assume that \ref{ass:entries_cdts} and \ref{ass:C_N} hold. Then, if $Z_{i,j}$ are standard complex Gaussian, or if $Z_{i,j}$ are standard real Gaussian and $C_N$ is real, we have
$$\Lambda_m(T)\cvweak[N,n\to\infty] \gNr\left(0,\sigma^2\int x^2\dd\nu_C(x) I_m\right)$$
where $\sigma^2=2$ in real Gaussian case, and $\sigma^2=1$ in complex Gaussian case.
\end{corol}

If $C_N=I_N$, then we can see that
$$\theta_j=1+\frac{1}{N}\sum_{k\ne j}\frac{t_k}{t_j-t_k}.$$
In general, $\theta_j$ has not a closed expression. However, $\theta_j$ can be expressed as a power series of $N^{-1}\sum_{k\ne j}\frac{t_k}{t_j-t_k}$.

\begin{prop}\label{prop:power_series} For a fixed $N$, let $m_k:=\tr C_N^k/N$, let the coefficients $B_0,B_1,\dots$ be defined by the recurrent formula
$$\begin{cases}B_0=m_1 \\ \sum_{k_1+\dots+k_{n+1}=n}\prod_{\ell=1}^{n+1}B_{k_\ell}=m_{n+1}+\sum_{j=1}^{n}m_{n-j+1}\left(\sum_{k_1+\dots+k_j=j}\prod_{\ell=1}^j B_{k_\ell}\right) & \text{ for }n\ge 1.\end{cases}$$
Suppose that not all the eigenvalues of $C_N$ are zero. Then the power series 
$$\theta(z):=B_0+B_1z+B_2z^2+\cdots$$
is the solution of the equation
$$G(\theta(z),z)=1.$$
Its radius of convergence $R$ satisfies
$$R\ge\inf \left\{|z|\tq \exists \theta\in\C \,\text{s.t.}\, G(\theta,z)=1,\quad \frac{\partial G(\theta,z)}{\partial \theta}=0\right\},$$
where we make the convention that $\inf\emptyset = +\infty$.
\end{prop}
From this proposition, we have, under the conditions \ref{ass:C_N}, \ref{ass:eigenvalue_conv_Gamma} and \ref{ass:spectral_gap}, for large enough $N,n$,
$$\theta_j=B_0+B_1\frac{1}{N}\sum_{k\ne j}\frac{t_k}{t_j-t_k}+B_2\left(\frac{1}{N}\sum_{k\ne j}\frac{t_k}{t_j-t_k}\right)^2+\cdots.$$
By the recurrent formula, we obtain
$$B_0=m_1, \quad B_1=\frac{m_2}{m_1},\quad B_2=\frac{m_3}{m_2^2}-\frac{m_2^2}{m_1^3}, \quad\dots$$
So we have
$$\theta_j=m_1+\frac{m_2}{m_1}\left(\frac{1}{N}\sum_{k\ne j}\frac{t_k}{t_j-t_k}\right)+\left(\frac{m_3}{m_2^2}-\frac{m_2^2}{m_1^3}\right)\left(\frac{1}{N}\sum_{k\ne j}\frac{t_k}{t_j-t_k}\right)^2+\cdots\,.$$

\begin{Rq}\label{rq:order_theta_j_m1} In \cite[Example~2.3]{merlevede2019unbounded}, we have given the various orders of $\sqrt{N}(\theta_j-1)$ when $C_N=I_N$, so in general $\theta_j$ can not be replaced by a finite form. However in some particular cases, we can replace $\theta_j$ by a partial sum of its Taylor's expansion. For example, when $\Gamma_n$ satisfies \ref{ass:contr_trace_1} below, we have $\sqrt{N}(\theta_j-m_1)\to 0$. Thus we can replace $\theta_j$ by $m_1$. One can see that the model in \cite{zhang2018clt} is in this case when their $\mathbf \Pi=\mathbf I$ (Theorem~3 of \cite{zhang2018clt}), because their major spiked population eigenvalues are asymptotically $\lambda_k\sim 4T^2/(\pi(2k-1))^2$ as $T\to\infty$, where $T$ denotes the dimension (Lemma~1 and 2 of \cite{zhang2018clt}). That is, with our notations, 
$$t_k=\frac{1}{2n^2(1+\cos(2(n+1-k)\pi/(2n+1)))}, \quad a_k=\frac{4}{\pi^2(2k-1)^2}.$$
And by calculating $n^{-1/2}\tr (\mathbf C^*\mathbf C)/n^2$ where $\mathbf C$ is defined in (2.4) of \cite{zhang2018clt}, we have 
$$\frac{1}{\sqrt{n}}\sum_{k=1}^n t_k=O(1/\sqrt{n}).$$

Similarly, when $\Gamma_n$ satisfies \ref{ass:contr_trace_2} below, $\theta_j$ can be replaced by 
$$m_1+\frac{m_2}{m_1}\left(\frac{1}{N}\sum_{k\ne j}\frac{t_k}{t_j-t_k}\right).$$
\end{Rq}

\subsection{CLT for largest eigenvalues: Some generalizations} \label{subsec:generalclt}
In this section we generalize the CLT to non-diagonal $\Gamma_n$. We continue to assume the other assumptions, and moreover, we assume that one of the two following assumptions holds:

\begin{enumerate}[leftmargin=*, label={\bf A\arabic{enumi}}]
\setcounter{enumi}{8}
    \item \label{ass:contr_trace_1}
    $$\int x\dd\mu^{\Gamma_n}(x)=\frac{\tr \Gamma_n}{n} = o(1/\sqrt{n}) \quad \text{and}\quad \E|Z_{i,j}|^4<\infty.$$
    \item \label{ass:contr_trace_2}
    $$\int x^2\dd\mu^{\Gamma_n}(x)=\frac{\tr \Gamma_n^2}{n} = o(1/\sqrt{n}) \quad \text{and}\quad \E|Z_{i,j}|^8<\infty.$$
\end{enumerate}

\begin{Rq} \label{rq:comment_ass} Under \ref{ass:eigenvalue_conv_Gamma}, because almost all the eigenvalues of $\Gamma_n$ (except for at most a finite number of them) are smaller than $1$, the condition \ref{ass:contr_trace_1} is stronger than \ref{ass:contr_trace_2}. They are some indicators who measure the degree of concentration of the eigenvalues near zero. If $\Gamma_n$ satisfies \ref{ass:contr_trace_1}, then its eigenvalues are more concentrated near $0$ than the case where it just satisfies \ref{ass:contr_trace_2}.

From Proposition~\ref{prop:decay_moments}, we can see that for $\rho\in(-3/4,0)$, the normalized Toeplitz matrix $\tilde T_n$ satisfies \ref{ass:contr_trace_2}, and for $\rho\in(-1/2,0)$,  $\tilde T_n$ satisfies \ref{ass:contr_trace_1}. 
\end{Rq}

\begin{Th}\label{th:joint_fluct} Under \ref{ass:entries_cdts}, \ref{ass:C_N}, \ref{ass:eigenvalue_conv_Gamma}, \ref{ass:C_diagonal}, \ref{ass:spectral_gap}, and either \ref{ass:contr_trace_1} or \ref{ass:contr_trace_2},  we have
\begin{equation}
    \dlp(\Lambda_m(\Gamma_N),\gNr(0,\Sigma_m^{(N)}))\xrightarrow[N,n\to\infty]{} 0\,, \label{eq:th_joint_fluct}
\end{equation}
Where $\Sigma_m^{(N)}=\frac{\tr C_N^2}{N}(I_m+(\sigma_{i,j})_{i,j=1}^m)$ with
\begin{equation}
\sigma_{i,j}=(\E|Z_{1,1}|^4-|\E Z_{1,1}^2|^2-2)\sum_{k=1}^n |u_{i,k}|^2|u_{j,k}|^2+\left|\sum_{k=1}^n u_{i,k}u_{j,k}\right|^2|\E Z_{1,1}^2|^2. \label{eq:sigma_ij}    
\end{equation}
and $u_j:=(u_{j,1},\dots,u_{j,n})^\tran$ is a normalized eigenvector associated with $\lambda_j(\Gamma_n)$.
\end{Th}

\begin{Rq}In view of the expression \eqref{eq:sigma_ij}, it is not clear that the covariance matrix $\Sigma_m^{(N)}$ converges. In order to avoid any cumbersome assumption enforcing this convergence, we express the CLT with the help of L\'evy-Prokhorov’s distance. If however it happens that $\Sigma_m^{(N)}$ converges to some matrix $\Sigma_m$, then we conclude the CLT in the following usual form
$$\Lambda_m(\Gamma_N)\cvweak[N,n\to\infty] \gNr(0,\Sigma_m).$$
\end{Rq}

From Proposition~\ref{prop:asympt_eigenvect} and Theorem~\ref{th:T_Tprime}, if $T_n$ is a Toeplitz matrix satisfying \eqref{eq:LM_gamma} or \eqref{eq:LM_f}, we can see that the eigenvectors $u_j$ of $T_n$ are delocalized, i.e. $\|u_j\|_\infty=o(1)$. So we have
$$\sum_{k=1}^n |u_{i,k}|^2|u_{j,k}|^2=o(1).$$
Also because $T_n$ are real, we have
$$\sum_{k=1}^n u_{i,k}u_{j,k}=\delta_{ij}\,.$$
So we get the following result:

\begin{corol}\label{th:main_fluct} Let $T_n$ be a sequence of Toeplitz matrices satisfying \eqref{eq:LM_gamma} or \eqref{eq:LM_f}. Let $S_N(T)$ be defined as before. Suppose that \ref{ass:entries_cdts}, \ref{ass:C_N}, \ref{ass:C_diagonal} hold. If one of the following is satisfied:
\begin{enumerate}
    \item The parameter $\rho$ belongs to $(0,1/2)$ and $\E |Z_{i,j}|^4<\infty$;
    \item The parameter $\rho$ belongs to $[1/2,3/4)$ and $\E |Z_{i,j}|^8<\infty$,
\end{enumerate}
then we have
$$\Lambda_m(T_N) \cvweak[N,n\to\infty] \gNr\left(0,(1+|\E Z_{i,j}^2|^2)\int x^2\dd\nu_C I_m\right)\,.$$
\end{corol}

\section{Proofs of the theorems on Toeplitz matrices}
\label{sec:proof_toeplitz}
In Section~\ref{subsec:preparation_toep}, \ref{subsec:proof_simple_eigen} and \ref{subsec:proof_eigenvect} we focus on Toeplitz matrices satisfying \eqref{eq:LM_gamma}. In Section~\ref{subsec:proof_spect_density} we treat the Toeplitz matrices satisfying \eqref{eq:LM_f}. And in Section~\ref{subsec:proof_decay_moments} we prove Proposition~\ref{prop:decay_moments} for Toeplitz matrices satisfying either \eqref{eq:LM_gamma} or \eqref{eq:LM_f}. 
\subsection{Some preparation} \label{subsec:preparation_toep}
Let $T_n$ satisfy \eqref{eq:LM_gamma}. Note that by the definition of slowly varying function, $\gamma(h)$ is positive for $h$ sufficiently large.

For $p\in [1,\infty]$, and for $n$ sufficiently large such that $\gamma(n)\ne 0$, we define a finite-rank operator $\K_n^{(\gamma)}$ acting on $L^p(0,1)$ by
\begin{equation}(\K_n^{(\gamma)}f)(x)=\int_0^1 \frac{\gamma(\lfloor nx\rfloor-\lfloor ny\rfloor)}{\gamma(n)}f(y)\dd y\,. \label{eq:K_N_c}
\end{equation}

The operator $\K^{(\rho)}$ in \eqref{eq:defi_K_rho} is also well-defined for any $f\in L^p(0,1)$ by the integral formula:
\begin{equation}
(\K^{(\rho)}f)(x)=\int_0^1 |x-y|^{-\rho} f(y)\dd y\,.
\end{equation}
The operators $\K_n^{(\gamma)}$ and $\K^{(\rho)}$ acting on $L^p(0,1)$ are bounded, see \cite[Lemma~5.4]{merlevede2019unbounded}. Moreover, from Lemma~5.4 of \cite{merlevede2019unbounded}, we have the convergence 
\begin{equation}
    \|\K_n^{(\gamma)}-\K^{(\rho)}\|_p\to 0 \quad \text{as}\quad n\to\infty, \quad \forall p\in[1,\infty]\,. \label{eq:proof:cv_operator}
\end{equation}
The convergence \eqref{eq:proof:cv_operator} has many useful consequences in this proof. The first consequence is that the operator $\K^{(\rho)}$ is compact on $L^p(0,1)$ for any $p\in[1,\infty]$. 

For each $p\in[1,\infty]$, $\K^{(\rho)}$ (resp. $\K_n^{(\gamma)}$) has its spectrum as an operator acting on $L^p(0,1)$. The following proposition shows that its non-zero eigenvalues and the associated eigenfunctions are invariant as $p$ changes.

\begin{prop} \label{th:WH_spect_invar}
The non-zero eigenvalues and the associated eigenfunctions of $\K^{(\rho)}:L^p(0,1)\rightarrow L^p(0,1)$ and $\K_n^{(\gamma)}:L^p(0,1)\rightarrow L^p(0,1)$ do not change when $p$ runs across $[1,\infty]$. 
\end{prop}
\begin{proof} 
We only prove the result for $\K^{(\rho)}$. The same argument applies to $\K_n^{(\gamma)}$.

We only need to prove that, for any $p\in [1,\infty)$, the operator $\K^{(\rho)}:L^p(0,1)\rightarrow L^p(0,1)$ has the same non-zero eigenvalues and associated eigenfunctions as $\K^{(\rho)}:L^\infty(0,1)\rightarrow L^\infty(0,1)$. As we have already noticed that $\K^{(\rho)}$ is compact on $L^p(0,1)$ and on $L^\infty(0,1)$, the desired result is a direct application of Theorem~4.2.15 in \cite{davies2007linear}. 

Indeed, we recall that two Banach spaces $B_1$ and $B_2$ or their associated norms are said to be compatible if $B =  B_1 \cap B_2$ is dense in each of them, and the following condition is satisfied: if $f_n\in B$, $\|f_n-f\|_{B_1} \to 0$ and  $\|f_n -g\|_{B_2} \to 0$, then $f = g \in B$. The operators $A_i: B_i\rightarrow C_i$ with $i=1,2$ and $C_1, C_2$ two Banach spaces, are said to be consistent if $A_1f = A_2f$ for all $f \in B_1 \cap B_2$. Then we can verify that $L^p(0,1)$ and $L^\infty(0,1)$ are compatible, and $\K^{(\rho)}$ defined by an integral formula is obviously consistent. Then Theorem~4.2.15 in \cite{davies2007linear} applies.
\end{proof}

According to the above proposition, when we talk about the non-zero eigenvalues and the associated eigenfunctions of these operators, we do not need to specify the space. 

\subsection{Proof of Proposition~\ref{prop:simplicity_eigen}} \label{subsec:proof_simple_eigen}
Let $\lambda\ne 0$ be an eigenvalue of $\K^{(\rho)}$ and $f$ be an associated eigenfunction. We now prove that $f$ is continuous on $[0,1]$.

Note that $f$ satisfies the equation
$$f(x)=\lambda^{-1}\int_0^1 |y-x|^{-\rho} f(y)\dd y\,,$$
and from Proposition~\ref{th:WH_spect_invar}, $f$ also belongs to $L^\infty(0,1)$. So for any $x_0\in[0,1]$, one has
$$|f(x)-f(x_0)|\le \lambda^{-1}\|f\|_\infty\int_0^1\left| |y-x|^{-\rho}- |y-x_0|^{-\rho} \right|\dd y\le \lambda^{-1}\|f\|_\infty\int_{-1}^1\left| |y-x+x_0|^{-\rho}- |y|^{-\rho} \right|\dd y\,,$$
and the integral on the RHS tends to $0$ when $|x-x_0|\to 0$.

We now prove that all non-zero eigenvalues of $\K^{(\rho)}$ are simple. We need the following key lemma. It says that any normalized eigenfunction of $\K^{(\rho)}$ associated with a non-zero eigenvalue, taken at $x=1$, has the absolute value $\sqrt{1-\rho}$. 

\begin{lemma}\label{lem:extreme_value_eigenfunc} Let $\lambda>0$ be a non-zero eigenvalue of $\K^{(\rho)}$, and let $f$ be a normalized eigenfunction associated with $\lambda$. Then $f$ satisfies
\begin{equation}
    |f(1)|=\sqrt{1-\rho}\,. \label{eq:extreme_value_eigenfun}
\end{equation}
\end{lemma}

A result similar to the above lemma first appeared in \cite{rao1976eigenvalues} for a general but square integrable kernel $k(x-y)$, see Theorem~3 of \cite{rao1976eigenvalues}. Note that thanks to the explicit formula of $\K^{(\rho)}$, the result of Lemma~\ref{lem:extreme_value_eigenfunc} is stronger than \cite{rao1976eigenvalues}. Directly using Theorem~3 of \cite{rao1976eigenvalues}, we can only conclude that for $\rho\in(-1/2,0)$, for any non-zero eigenvalue $\lambda$ of $\K^{(\rho)}$, there exists a group of orthonormal eigenfunctions $f_{\lambda,1},\dots,f_{\lambda,m}$ associated with $\lambda$, where $m$ is the multiplicity of $\lambda$, such that
$$|f_{\lambda,i}(1)|=\sqrt{1-\rho}\,.$$
However we will notice that this result is not sufficient to prove the simplicity of eigenvalues.

Whenever Lemma~\ref{lem:extreme_value_eigenfunc} is proved, we can prove the simplicity of any non-zero eigenvalue $\lambda$ of $\K^{(\rho)}$ by contradiction. Assume to the contrary that $\lambda>0$ had multiplicity $m\ge 2$, then we could choose two orthonormal eigenfunctions $f_{\lambda,1}, f_{\lambda,2}$ associated with $\lambda$. From Lemma~\ref{eq:extreme_value_eigenfun}, without loss of generality we can assume that $f_{\lambda,1}(1)=f_{\lambda,2}(1)=\sqrt{1-\rho}$. Then the function 
$$f_{\lambda}:=\frac{1}{\sqrt{2}} f_{\lambda,1}+\frac{1}{\sqrt{2}} f_{\lambda,2}$$
is also a normalized eigenfunction of $\K^{(\rho)}$. But this function satisfies
$$f_{\lambda}(1)=\sqrt{2(1-\rho)}\ne \sqrt{1-\rho}\,,$$
which is a contradiction to Lemma~\ref{lem:extreme_value_eigenfunc}.

Thus it remains to prove Lemma~\ref{lem:extreme_value_eigenfunc}.

\begin{proof}[Proof of Lemma~\ref{lem:extreme_value_eigenfunc}] We follow the outline of the proof in \cite{rao1976eigenvalues}. For any $\tau>0$, we define $\K^{(\rho)}_\tau$ the operator on $L^2(0,\tau)$ by
\begin{equation}
(\K^{(\rho)}_\tau f)(x)=\int_0^\tau |x-y|^{-\rho} f(y)\dd y\, . \label{eq:def_Kappa_tau}
\end{equation}
By a change of variable, it is easy to see that a function $f\in L^2(0,1)$ is an eigenfunction of $\K^{(\rho)}$ associated with an eigenvalue $\lambda$ if and only if $f(\frac{\cdot}{\tau})$ is an eigenfunction of $\K^{(\rho)}_\tau$ associated with the eigenvalue $\lambda \tau^{1-\rho}$. By this fact, a positive number $\lambda$ is an eigenvalue of $\K^{(\rho)}$ with multiplicity $m$ if and only if $\lambda \tau^{1-\rho}$ is an eigenvalue of $\K^{(\rho)}_\tau$ with the same multiplicity $m$ for all $\tau>0$. 

Suppose that $f$ is a normalized eigenfunction of $\K^{(\rho)}$ associated with non-zero eigenvalue $\lambda>0$. Then for any $\tau>1$ we have the following two equations
\begin{equation}
    \lambda\tau^{1-\rho} f\left(\frac{x}{\tau}\right) = \int_0^\tau |x-y|^{-\rho} f\left(\frac{y}{\tau}\right)\dd y \, , \text{ for } x\in(0,\tau) \label{eq:proof_lemma_ex_va_1}
\end{equation}
and
\begin{equation}
    \lambda\overline{f(y)} = \int_0^1 |x-y|^{-\rho} \overline{f(x)}\dd x \, , \text{ for } y\in(0,1).\label{eq:proof_lemma_ex_va_2}
\end{equation}
We define the function $g$ on $[0,\infty)$ by
$$g(y)=\frac{1}{\lambda}\int_0^1 |x-y|^{-\rho} f(x)\dd x\,, \text{ for } y\in[0,\infty),$$
then $g$ is a continuous extension of $f$ on $[0,\infty)$. Multiply the two sides of \eqref{eq:proof_lemma_ex_va_1} by $\overline{f(x)}$, and integrate for $x\in [0,1]$, we get
\begin{equation}
    \lambda\tau^{1-\rho} \int_0^1 f\left(\frac{x}{\tau}\right)\overline{f(x)}\dd x = \int_0^1\int_0^\tau |x-y|^{-\rho} f\left(\frac{y}{\tau}\right)\overline{f(x)}\dd y\dd x\,.
\end{equation}
Note that by the boundedness of $f$, Fubini Theorem applies to the RHS, thus changing the order of two integrations and taking into account the definition of $g$, we get
\begin{equation}
    \tau^{1-\rho} \int_0^1 f\left(\frac{x}{\tau}\right)\overline{f(x)}\dd x = \int_0^\tau \overline{g(y)} f\left(\frac{y}{\tau}\right)\dd y\,. \label{eq:proof_lemma_ex_va_3}
\end{equation}
It is easy to see from \eqref{eq:proof_lemma_ex_va_3} that
\begin{equation}
    \left(\frac{\tau^{1-\rho}-1}{\tau-1}\right)\int_0^1 f\left(\frac{x}{\tau}\right)\overline{f(x)}\dd x 
    = \frac{\int_1^\tau \overline{g(y)} f\left(\frac{y}{\tau}\right)\dd y }{\tau-1} \, . \label{eq:proof_lemma_ex_va_4}
\end{equation}
Letting $\tau\to 1^+$ on the two sides of \eqref{eq:proof_lemma_ex_va_4}, and noting that the continuity of $f$ on $[0,1]$ implies the uniform convergence of $f(\frac{x}{\tau})$ to $f(x)$, we get
\begin{equation}
    1-\rho = |f(1)|^2
\end{equation}
and the result follows.
\end{proof}

\subsection{Proof of Proposition~\ref{prop:asympt_eigenvect}} \label{subsec:proof_eigenvect}
Let $j\ge 1$ be an integer. By the proof of Theorem~2.3 in \cite{merlevede2019unbounded}, we have $\lambda_j(\K^{(\rho)})>0$. Also by Proposition~\ref{prop:simplicity_eigen}, we have $\lambda_{j-1}(\K^{(\rho)})>\lambda_j(\K^{(\rho)})>\lambda_{j+1}(\K^{(\rho)})$. Let $f_j$ be a normalized eigenfunction associated with $\lambda_j(\K^{(\rho)})$. In the sequel, we shall rely on the spectral projections (to be defined later) to construct an eigenvector of $T_n$ associated with $\lambda_j(T_n)$ and prove that such an eigenvector approximates $f_j$ in the sense of  \eqref{eq:th:asympt_eig_vec}.

Let $\varepsilon=\frac{1}{2}\min(\lambda_{j-1}(\K^{(\rho)})-\lambda_j(\K^{(\rho)}), \lambda_{j}(\K^{(\rho)})-\lambda_{j+1}(\K^{(\rho)}))$ and $\mathscr{C}$ be the circle centered at $\lambda_j(\K^{(\rho)})$ and of radius $\varepsilon$ on complex plane. We take $n$ sufficiently large such that  $\|\K_n^{(\gamma)}-\K^{(\rho)}\|_\infty<\varepsilon$. So we have  $|\lambda_k(\K_n^{(\gamma)})-\lambda_k(\K^{(\rho)})|<\varepsilon$ for all $1\le k\le n$, which implies that only the eigenvalues $\lambda_j(\K_n^{(\gamma)})$ and $\lambda_j(\K^{(\rho)})$ are enclosed by $\mathscr C$ and all the other eigenvalues are outside $\mathscr{C}$. We define the spectral projections
\begin{equation}
    P_n:=\frac{1}{2\pi \im}\int_{\mathscr C} (z-\K_n^{(\gamma)})^{-1}\dd z\quad \text{and}\quad P:=\frac{1}{2\pi \im}\int_{\mathscr C} (z-\K^{(\rho)})^{-1}\dd z.
\end{equation}
By Riesz decomposition Theorem (c.f. for example  \cite[Theorem~1.5.4 and Theorem~4.3.19]{davies2007linear}), $P_n$ (resp. $P$) is a projection onto the eigenspace of $\K_n^{(\gamma)}$ (resp. $\K^{(\rho)}$)) corresponding to $\lambda_j(\K_n^{(\gamma)})$ (resp. $\lambda_j(\K^{(\rho)})$).  To those who are unfamiliar with Riesz' Theorem, we explain the arguments with $\K^{(\rho)}$ and $P$. Indeed, from Riesz' Theorem, $P$ is a finite rank projection which commutes with $\K_{\rho}$. Let $\range(P)$ be the range of $P$, then $\range(P)$ is an invariant space of $\K^{(\rho)}$ (due to the commutativity of the projection $P$ and $\K^{(\rho)}$), and the restriction of $\K^{(\rho)}$ to $\range(P)$ is self-adjoint (because $\K^{(\rho)}$ is self-adjoint) and has spectrum $\{\lambda_j(\K^{(\rho)})\}$, then from the finite dimensional linear algebra, $\range(P)$ is spanned by the eigenfunctions of $\K^{(\rho)}$ associated with $\lambda_j(\K^{(\rho)})$. Therefore, recall that $f_j$ is a normalized eigenfunction of $\K^{(\rho)}$ associated with $\lambda_j(\K^{(\rho)})$, we have $Pf_j=f_j$. The same argument shows that $P_n$ is a projection to the eigenspace of $\K_n^{(\gamma)}$ and thus $P_nf_j$ is an eigenfunction of $\K_n^{(\gamma)}$ associated with $\lambda_j(\K_n^{(\gamma)})$, in condition that $P_nf_j\ne 0$.

We prove that $\|P_n-P\|_\infty\to 0$. Indeed we have
$$\|P_n-P\|_\infty\le \frac{1}{2\pi}\int_{\mathscr C}\|(z-\K_n^{(\gamma)})^{-1}-(z-\K^{(\rho)})^{-1}\|_\infty|\dd z|.$$
Thus the main task is to uniformly control $\|(z-\K_n^{(\gamma)})^{-1}-(z-\K^{(\rho)})^{-1}\|_\infty$ in term of $\|\K_n^{(\gamma)}-\K^{(\rho)}\|_\infty$ for $z\in\mathscr{C}$. As $(z-\K^{(\rho)})^{-1}$ is analytic outside of $\spec(\K^{(\rho)})$, there exists $K>0$ such that $\sup_{z\in\mathscr{C}}\|(z-\K^{(\rho)})^{-1}\|_\infty\le K$. Let $n$ be sufficiently large such that $\|\K_n^{(\gamma)}-\K^{(\rho)}\|_\infty<1/(2K)$. Then we have
$$\begin{aligned}\left\|(z-\K_n^{(\gamma)})^{-1}-(z-\K^{(\rho)})^{-1}\right\|_\infty &= \left\|(z-\K^{(\rho)}-\K_n^{(\gamma)} +\K^{(\rho)})^{-1}-(z-\K^{(\rho)})^{-1}\right\|_\infty       \\
&= \left\|(z-\K^{(\rho)})^{-1}\left[\left(I-(\K_n^{(\gamma)}-\K^{(\rho)})(z-\K^{(\rho)})^{-1}\right)^{-1}-I\right]\right\|_\infty  \\
&=\left\|(z-\K^{(\rho)})^{-1}\sum_{k=1}^\infty((\K_n^{(\gamma)}-\K^{(\rho)})(z-\K^{(\rho)})^{-1})^k\right\|_\infty \\
    &\le \sum_{k=1}^\infty \frac{K^{k+1}}{2^{k-1}K^{k-1}}\|\K_n^{(\gamma)}-\K^{(\rho)}\|_\infty^{k}\\
    &= 2K^2\|\K_n^{(\gamma)}-\K^{(\rho)}\|_\infty.\end{aligned}$$
Thus as $n\to\infty$ we have
$$\|P_n-P\|_\infty\le \frac{1}{2\pi}\int_{\mathscr{C}}\|(z-\K_n^{(\gamma)})^{-1}-(z-\K^{(\rho)})^{-1}\|_\infty|\dd z|\le {2\varepsilon K^2}\|\K_n^{(\gamma)}-\K^{(\rho)}\|_\infty \to 0.$$

From this convergence we have
\begin{equation}
    \|P_nf_j-f_j\|_\infty=\|P_nf_j-Pf_j\|_\infty\xrightarrow[n\to\infty]{} 0\,. \label{eq:proof:cv_eigenfunc}
\end{equation}
Then from \eqref{eq:proof:cv_eigenfunc} we obtain
\begin{equation}\|P_nf_j\|_2\xrightarrow[n\to\infty]{} \|f_j\|_2=1\,. \label{eq:proof:cv_eigenfunc_norm} \end{equation}
Combining \eqref{eq:proof:cv_eigenfunc} and \eqref{eq:proof:cv_eigenfunc_norm} we conclude
\begin{equation}\left\|\frac{P_nf_j}{\|P_nf_j\|_2}-f_j\right\|_\infty \xrightarrow[n\to\infty]{} 0\,. \label{eq:proof:cv_eigenfunc_normalized}\end{equation}

Notice that the range of $\K^{(\gamma)}_n$ consists of step functions
$$f(x)=\sum_{k=1}^n v_k\ind_{[\frac{k-1}{n},\frac{k}{n})}(x)\,,$$
so the eigenfunctions of $\K^{(\gamma)}_n$ must also have this form. Notice also that a $n$-dimensional normalized vector $v=(v_k)_{k=1}^n$ is an eigenvector of $T_n$ associated with $\lambda_j(T_n)$ if and only if the normalized function
\begin{equation}
    f(x)=\sqrt{n}\sum_{k=1}^n v_k\ind_{[\frac{k-1}{n},\frac{k}{n})}(x) \label{eq:proof:eigenfunc_vect}
\end{equation} is an eigenfunction of $\K_n^{(\gamma)}$ associated with $\lambda_j(\K_n^{(\gamma)})=\lambda_j(T_n)/(n\gamma(n))$. Since $Pf_j/\|Pf_j\|_2$ is a normalized eigenfunction of $\K_n^{(\gamma)}$, by the relation \eqref{eq:proof:eigenfunc_vect}, up to a change of sign, we have
\begin{equation}
    u_{j,k}=\frac{(P_nf_j)(\frac{k-1}{n})}{\sqrt{n}\|P_nf_j\|_2}, \quad \forall k=1,\dots,n.
\end{equation}
From \eqref{eq:proof:cv_eigenfunc_normalized} we get the desired result \eqref{eq:th:asympt_eig_vec}.

\subsection{Proof of Theorem~\ref{th:T_Tprime}} \label{subsec:proof_spect_density}
Let $T_n=(\gamma(i-j))$ be a $n\times n$ Toeplitz matrix with spectral density $\varphi$ satisfying \eqref{eq:LM_f}. Let $T'_n=(\gamma'(i-j))$ be a $n\times n$ Toeplitz matrix with spectral density $x\mapsto 1/|x|^{1-\rho}, x\in[-\pi,\pi]$. Let $D_n(x):=\frac{\sin((n+1/2)x)}{\sin(x/2)}$ be the Dirichlet kernel. From the theory of Toeplitz matrices, we have
\begin{equation}
    \begin{aligned}\left\|\frac{T_n}{n^{1-\rho}L_2(n)}-\frac{T'_n}{n^{1-\rho}}\right\| &\le \sup_{x\in [-\pi,\pi]}\left|\frac{1}{n^{1-\rho}L_2(n)}\sum_{k=-n}^n\gamma(k)e^{\im k x} - \sum_{k=-n}^n\frac{1}{n^{1-\rho}}\gamma'(k)e^{\im k x}\right| \\
    &\le \sup_{x\in [-\pi,\pi]} \frac{1}{2\pi n^{1-\rho}}\int_{-\pi}^{\pi} \frac{1}{|y|^{1-\rho}}\left|\frac{L_2(|y|^{-1})}{L_2(n)}-1\right| |D_n(x-y)|\dd y.
    %&= \sup_{x\in [-n\pi,n\pi]} \frac{1}{2\pi n}\left|\int_{-n\pi}^{n\pi} \frac{1}{|y|^{1-\rho}}\left(\frac{L_2(n|y|^{-1})}{L_2(n)}-1\right) D_n(\frac{x-y}{n})\dd y\right|.
\end{aligned} \label{eq:ttpr_spn_ctrl}
\end{equation}
By the inequality (2.2.1) of \cite{pipiras2017long}, for a certain $0<\delta<\min(\rho/2, (1-\rho)/2)$, there exists $0<\eta<1$ such that
\begin{equation}
    \frac{L_2(|u|^{-1})}{L_2(|v|^{-1})}\le 2\max\left(\left|\frac{u}{v}\right|^\delta, \left|\frac{v}{u}\right|^\delta\right), \quad \forall u,v\in (-\eta,\eta). \label{eq:bound_L2_ratio}
\end{equation}
Then because $L_2$ is locally bounded, we have
$$\begin{aligned}\sup_{x\in [-\pi,\pi]} \frac{1}{n^{1-\rho}L_2(n)}\int_{\eta<|y|\le \pi} \frac{L_2(|y|^{-1})}{|y|^{1-\rho}}|D_n(x-y)|\dd y &\le \frac{K}{n^{1-\rho}L_2(n)}\sup_{x}\int_{-\pi}^\pi |D_n(x-y)|\dd y \\
    & = O\left(\frac{\log n}{n^{1-\rho}L_2(n)}\right)\xrightarrow[n\to\infty]{} 0\,.\end{aligned}$$
The same argument also gives
$$\sup_{x\in [-\pi,\pi]} \frac{1}{n^{1-\rho}}\int_{\eta<|y|\le \pi} \frac{1}{|y|^{1-\rho}}|D_n(x-y)|\dd y \xrightarrow[n\to\infty]{} 0\,.$$
Combining the last two inequalities, and using the triangle inequality, we get
$$\sup_{x\in [-\pi,\pi]} \frac{1}{n^{1-\rho}}\int_{\eta<|y|\le \pi} \frac{1}{|y|^{1-\rho}}\left|\frac{L_2(|y|^{-1})}{L_2(n)}-1\right| |D_n(x-y)|\dd y\xrightarrow[n\to\infty]{}0\,.$$
Thus in order to prove that \eqref{eq:ttpr_spn_ctrl} tends to $0$, we only need to prove that
$$\sup_{x\in [-\pi,\pi]} \frac{1}{n^{1-\rho}}\int_{-\eta}^{\eta} \frac{1}{|y|^{1-\rho}}\left|\frac{L_2(|y|^{-1})}{L_2(n)}-1\right| |D_n(x-y)|\dd y\to 0\,.$$
By changing variables we write
\begin{equation}
   \begin{aligned} &\sup_{x\in [-\pi,\pi]}\frac{1}{n^{1-\rho}}\int_{-\eta}^{\eta} \frac{1}{|y|^{1-\rho}}\left|\frac{L_2(|y|^{-1})}{L_2(n)}-1\right| |D_n(x-y)|\dd y \\
   =& \sup_{x\in [-n\pi,n\pi]}\int_{-n\eta}^{n\eta} \frac{1}{|y|^{1-\rho}}\left|\frac{L_2(n|y|^{-1})}{L_2(n)}-1\right| \left|\frac{1}{n}D_n(\frac{x-y}{n})\right|\dd y\,. \end{aligned} \label{eq:density_change_variable}
\end{equation}
Let 
$$\Delta(x):=\min(|x|^{-1}, 1).$$
From the properties of Dirichlet kernels, there exists a constant $K>0$ such that for any $-n\pi\le x\le n\pi$ and $-n\eta\le y\le n\eta$,
\begin{equation}
   n^{-1}|D_n((x-y)/n)|\le K\Delta(x-y). \label{eq:D_n_Delta}
\end{equation}
For any $\varepsilon>0$, let $A>1$ be a large enough positive number to be determined afterwards. Then
$$\begin{aligned} & \int_{-n\eta}^{n\eta} \frac{1}{|y|^{1-\rho}}\left|\frac{L_2(n|y|^{-1})}{L_2(n)}-1\right| \left|\frac{1}{n}D_n(\frac{x-y}{n})\right|\dd y \\
\le & K\int_{-n\eta}^{n\eta} \frac{1}{|y|^{1-\rho}}\left|\frac{L_2(n|y|^{-1})}{L_2(n)}-1\right| \Delta(x-y)\dd y \\ 
    = & K\left(\int_{-\frac{1}{A}}^\frac{1}{A}+\int_{\frac{1}{A}<|y|\le A}+\int_{A<|y|\le n\eta}\right)\frac{1}{|y|^{1-\rho}}\left|\frac{L_2(n|y|^{-1})}{L_2(n)}-1\right| \Delta(x-y)\dd y \\
    =: & K(P_1(x)+P_2(x)+P_3(x)).\end{aligned}$$
Because $\Delta$ is bounded by $1$, and also by \eqref{eq:bound_L2_ratio}, we have
$$\sup_x P_1(x)\le \int_{-\frac{1}{A}}^\frac{1}{A}\frac{1}{|y|^{1-\rho}}\left(\frac{L_2(n|y|^{-1})}{L_2(n)}+1\right)\dd y\le \int_{-\frac{1}{A}}^\frac{1}{A}\frac{3}{|y|^{1-\rho+\delta}}\dd y=\frac{6}{(\rho-\delta)A^{\rho-\delta}}.$$
Using \eqref{eq:bound_L2_ratio} and Young's convolution inequality, let $\frac{1}{p}=\frac{1-\rho-\delta}{2}, \frac{1}{q}=\frac{1+\rho+\delta}{2}$, we have
$$\sup_{x}P_3(x)\le \sup_{x}\int_{A<|y|\le n\eta}\frac{3}{|y|^{1-\rho-\delta}}\Delta(x-y)\dd y\le \left|2\int_A^\infty \frac{\dd y}{y^2}\right|^{1/p}\|\Delta\|_q\le \frac{2^{1/p}\|\Delta\|_q}{A^{1/p}}.$$
Let $A=\max(\varepsilon^{-p}, \varepsilon^{-\frac{1}{\rho-\delta}})$, then there exists $K>0$ such that
$$\sup_x (P_1(x)+P_3(x))\le K\varepsilon\,.$$
By the uniform convergence theorem for slowly varying functions (see for example (2.2.4) of \cite{pipiras2017long}), for large enough $n$, we have
$$\sup_{|y|\in[A^{-1},A]}\left|\frac{L_2(n|y|^{-1})}{L_2(n)}-1\right|\le \frac{\varepsilon}{(4\log A)^{1-\rho}}.$$
Then using Young's inequality again, we have
$$\sup_{x}P_2(x)\le \frac{\varepsilon}{(4\log A)^{1-\rho}}\left|2\int_{1/A}^A \frac{\dd y}{y}\right|^{1-\rho}\|\Delta\|_{\frac{1}{\rho}}= \|\Delta\|_{\frac{1}{\rho}}\varepsilon\,.$$
Therefore, there exists $K>0$ such that, for any $\varepsilon>0$, for large enough $n$, 
$$\sup_{x\in [-\pi,\pi]} \frac{1}{n^{1-\rho}}\int_{-\pi}^{\pi} \frac{1}{|y|^{1-\rho}}\left|\frac{L_2(|y|^{-1})}{L_2(n)}-1\right| |D_n(x-y)|\dd y < K\varepsilon$$
and the proof of \eqref{eq:T_Tprime_norm} is complete. 

The convergence \eqref{eq:Toep_f_asy_eval} is an immediate consequence of \eqref{eq:T_Tprime_norm}, \eqref{eq:Toep_asy_eval} and \eqref{eq:relation_L1L2}. To prove \eqref{eq:T_Tprime_evect}, using the spectral projections and repeat the same procedure as in the proof of Proposition~\ref{prop:asympt_eigenvect} with $L^2\to L^2$ norm, the result then follows.

\subsection{Proof of Proposition~\ref{prop:decay_moments}} \label{subsec:proof_decay_moments}
First we prove Item \ref{enum:th:tech_assumpt_1}. Let $T_n$ satisfies \eqref{eq:LM_gamma} or \eqref{eq:LM_f} and assume $\rho\in(0,1/2)$. From \eqref{eq:Toep_asy_eval} or \eqref{eq:Toep_f_asy_eval}, there exists a constant $K>0$ such that
$$\lambda_1(T_n)\sim Kn^{1-\rho}L_i(n)$$
with $i=1, 2$. Since $n^{1/2-\rho}L_i(n)\to \infty$, we have
$$\frac{\tr T_n}{\sqrt{n}\lambda_1(T_n)}=\frac{\sqrt{n} \gamma(0)}{\lambda_1(T_n)}\sim \frac{\gamma(0)}{Kn^{1/2-\rho}L_i(n)}\to 0\,.$$

Then we prove Item \ref{enum:th:tech_assumpt_2}. Let $T_n$ satisfy \eqref{eq:LM_gamma} with $\rho\in[1/2,3/4)$. Note that
$$\tr T_n^2=\|T_n\|_F^2\le 2n\sum_{k=0}^n|\gamma(k)|^2\,.$$
Also from \eqref{eq:Toep_asy_eval}, we have $\lambda_1^2(T_n)\sim \lambda_1^2(\K^{(\rho)})n^{2+2\rho}L_1^2(n)\gg n^{1/2+\epsilon}$ for some $\epsilon>0$, where for two sequences of positive numbers $(x_n)_n$ and $(y_n)_n$, the notation $x_n\gg y_n$ means that $y_n/x_n\to 0$. We then have
$$\frac{\tr T_n^2}{\sqrt{n}\lambda_1^2(T_n)}\ll \frac{\sum_{k=0}^n|\gamma(k)|^2}{n^\epsilon}\le \frac{1}{n^{\epsilon/2}}\left(\gamma(0)+\sum_{k=1}^n \frac{|\gamma(k)|^2}{k^{\epsilon/2}}\right)\,.$$
Since $|\gamma(k)|^2\sim L_1^2(k) k^{2\rho}\le L_1^2(k) k^{-1}$, it follows that 
$$\sum_{k=1}^\infty \frac{|\gamma(k)|^2}{k^{\epsilon/2}}\le \sum_{k=1}^\infty \frac{L_1^2(k)}{k^{1+\epsilon/2}} <\infty\,.$$
Hence
$$\frac{\tr \tilde T_n^2}{\sqrt{n}}=\frac{\tr T_n^2}{\sqrt{n}\lambda_1^2(T_n)}\ll \frac{1}{n^{\epsilon/2}}\left(\gamma(0)+\sum_{k=1}^\infty \frac{L_1^2(k)}{k^{1+\epsilon/2}}\right)\to 0\,.$$
Now let $T_n$ satisfy \eqref{eq:LM_f} with $\rho\in[1/2,3/4)$. From Theorem~\ref{th:T_Tprime}, there exists $K>0$ such that $\|T_n\|\sim Kn^{1-\rho}L_2(n)$.  From the formula 
$$\gamma(k)=\frac{1}{2\pi}\int_{-\pi}^{\pi}\varphi(x)e^{-\im kx}\dd x\,,$$
we have
$$\tr T_n^2=\frac{n}{4\pi^2}\int_{-\pi}^{\pi}\int_{-\pi}^{\pi}\varphi(x)\varphi(y)F_n(x-y)\dd x\dd y$$
where $F_n(x)=\frac{\sin^2(\frac{nx}{2})}{n\sin^2(\frac{x}{2})}$ is the F\'ejer kernel. Let $1<r<\frac{2}{4\rho-1}$, and $p$ be such that $2/p+1/r=2$. Then we have $1<p<2$. From the form of $\varphi$, we have $\varphi\in L^p(-\pi,\pi)$. By Young's inequality,
$$\frac{\tr T_n^2}{\sqrt{n}\|T_n\|^2}\le \frac{\|\varphi\|_p^2\|F_n\|_r}{n^{3/2-2\rho}L_2^2(n)}\,.$$
Note that
$$\begin{aligned}\|F_n\|_r & = \frac{1}{2\pi n}\left|\int_{-\pi}^\pi \frac{\sin^{2r}(\frac{nx}{2})}{\sin^{2r}{(\frac{x}{2})}}\dd x\right|^{1/r}
=\frac{1}{2\pi n^{1+1/r}}\left|\int_{-\pi n}^{\pi n} \frac{\sin^{2r}(\frac{x}{2})}{\sin^{2r}{(\frac{x}{2n})}}\dd x\right|^{1/r} \\
& \le \frac{K}{n^{1/r-1}}\left|\int_{-\pi n}^{\pi n} \frac{\sin^{2r}(\frac{x}{2})}{x^{2r}}\dd x\right|^{1/r} \le \frac{K}{n^{1/r-1}}\left|\int_\R \frac{\sin^{2r}(\frac{x}{2})}{x^{2r}}\dd x\right|^{1/r}. \end{aligned}$$
Then
$$\frac{\tr T_n^2}{\sqrt{n}\|T\|^2}\le \frac{K}{n^{1/2+1/r-2\rho}L_2^2(n)}\to 0\,.$$

\section{Proofs of Proposition~\ref{prop:conv_eigen_SN} and~\ref{prop:power_series}}
\subsection{Proof of Proposition~\ref{prop:conv_eigen_SN}}
Let $j\ge 1$ be a fixed integer. We first prove the convergence of $\lambda_j(S_N(\Gamma))$ in probability. Suppose that $\Gamma_n=U\diag(t_1,\dots,t_n)U^*$ where $U=(u_1,\dots,u_n)$ is a unitary matrix whose columns are $u_1,\dots,u_n$. Recall that \ref{ass:eigenvalue_conv_Gamma} holds. For any $\varepsilon>0$ sufficiently small, let $\ell>j$ be the smallest integer such that $a_\ell < \varepsilon/(2K)$, where $K=\sup\|C_N\|(1+\sqrt{r})^2(1+\varepsilon)$.  let $n$ be large enough such that $|t_k-a_k|<\varepsilon/(2K)$ for $k=1,\dots,\ell+1$.

Let $\tilde \Gamma_n=U\diag(t_1,\dots,t_\ell,0,\dots,0)U^*$. Then we have
$$S_N(\Gamma)=S_N(\tilde \Gamma_n)+\frac{1}{N}C_N^{1/2}Z(\Gamma_n-\tilde \Gamma_n)Z^*C_N^{1/2}$$
As $\|\Gamma_n-\tilde \Gamma_n\|=t_{\ell+1}<\varepsilon/K$, from \cite{yin1988limit} we know that for large $N,n$, with high probability, 
$$\|\frac{1}{N}C_N^{1/2}Z(\Gamma_n-\tilde \Gamma_n)Z^*C_N^{1/2}\|\le \|C_N\|\|\Gamma_n-\tilde \Gamma_n\|(1+\sqrt{r})^2(1+\varepsilon)\le \varepsilon.$$
Thus from the stability of spectrum of Hermitian matrices,  with high probability, we have
$$|\lambda_j(S_N(\Gamma))-\lambda_j(S_N(\tilde\Gamma))|\le \varepsilon.$$
So we only need to prove that 
$$\lambda_j(S_N(\tilde \Gamma))\xrightarrow[N,n\to\infty]{\mathcal P} a_j\int x\dd\nu_C(x).$$

The matrix $S_N(\tilde \Gamma)$ has the same non-zero eigenvalues with the $\ell\times \ell$ matrix
$$\begin{aligned}A_\ell &:= \frac{1}{N}\diag(\sqrt{t_1},\dots,\sqrt{t_\ell})(u_1,\dots,u_\ell)^*Z^*C_NZ(u_1,\dots,u_\ell)\diag(\sqrt{t_1},\dots,\sqrt{t_\ell}) \\
    &= \left(\frac{1}{N}\sqrt{t_it_k}u_i^*Z^*C_NZu_k\right)_{i,k=1}^\ell \,.\end{aligned}$$
Then because $A_\ell$ is a fixed-dimensional matrix, it suffices to prove that each of its entries converges in probability. Note that
$$\begin{aligned} & \var \frac{1}{N}\sqrt{t_it_k}u_i^*Z^*C_NZu_k \\
= & \frac{t_it_k}{N^2}\left(\sum_{p=1}^n |u_{i,p}|^2|u_{k,p}|^2\sum_{p=1}^NC_{p,p}^2(\E|Z_{1,1}|^4-|\E Z_{1,1}^2|^2-2)+\left|\sum_{p=1}^n u_{i,p}u_{k,p}\right|^2|\E Z_{1,1}^2|^2\tr C_NC_N^\tran + \tr C_N^2\right)\end{aligned}$$
where $C_{p,p}$ are the diagonal entries of $C_N$. Because $\|C_N\|$ is uniformly bounded, we have
$$\max\left(\frac{\sum_{p=1}^NC_{p,p}^2}{N^2},\, \frac{\tr C_NC_N^\tran}{N^2}\right) \le \frac{\tr C_N^2}{N^2}=O(1/N).$$
Thus we have
$$\|A_\ell-\E A_\ell\|\xrightarrow[N,n\to\infty]{\mathcal P} 0.$$
Combine with the equality
$$\E A_\ell = \frac{\tr C_N}{N}\diag(t_1,\dots,t_\ell),$$
we obtain the convergence in probability.

Assume that $Z_{i,j}$ are standard real or complex Gaussian and prove the almost sure convergence. We argue similarly as the proof of Proposition~4.1 of \cite{merlevede2019unbounded}. Precisely we will prove that for any $\epsilon>0$,
\begin{equation}
\P\left(\left|\sqrt{\lambda_j(S_N(\Gamma))}-\E\sqrt{\lambda_j(S_N(\Gamma))}\right|>\epsilon\right)<2e^{-KN\epsilon^2} \label{eq:concentration}
\end{equation}
where $K$ is a constant. Indeed using  \cite[Theorem~5.6]{boucheron2013concentration} we only need to prove that the function
$$Z\mapsto \sqrt{\lambda_j(S_N(\Gamma))}$$
is $\frac{1}{\sqrt{KN}}$-Lipschitz with respect to the Frobenius norm $\|Z\|_F$. Let $Z, \hat Z$ be two $N\times n$ matrices, and let 
$$S=\frac{1}{N}C_N^{1/2}Z\Gamma_nZ^*C_N^{1/2},\quad \hat S=\frac{1}{N}C_N^{1/2}\hat Z\Gamma_n\hat Z^*C_N^{1/2}.$$
Then by Wielandt-Hoffmann inequality for singular values, we have
$$\sum_{j=1}^{\min(N,n)}\left|\sqrt{\lambda_j(S)}-\sqrt{\lambda_j(\hat S)}\right|^2\le \frac{1}{N}\|C_N^{1/2}(Z-\hat Z)\Gamma_n^{1/2}\|_F^2\le \frac{\|C_N\|\|\Gamma_n\|}{N}\|Z-\hat Z\|_F^2.$$
Thus we have
$$\left|\sqrt{\lambda_j(S)}-\sqrt{\lambda_j(\hat S)}\right|\le \frac{K}{\sqrt{N}}\|Z-\hat Z\|_F.$$
This proves the Lipschitz property and the concentration inequality \eqref{eq:concentration} holds. Then by Borel-Cantelli's lemma, we have
$$\sqrt{\lambda_j(S_N(\Gamma))}-\E\sqrt{\lambda_j(S_N(\Gamma))}\xrightarrow[N,n\to\infty]{\as}0.$$
Together with the convergence in probability
$$\lambda_j(S_N(\Gamma))\xrightarrow[N,n\to\infty]{\mathcal P}a_j\int x\dd\nu_C(x),$$
the almost sure convergence in the Gaussian case follows.

We now assume that the bound condition \ref{ass:bound_cdt} holds and prove the following lemma which will be useful in Section~\ref{sec:proof_diag_clt}. As a byproduct, this lemma implies the almost sure convergence of $\lambda_j(S_N(\Gamma))$.

\begin{lemma}\label{lemma:high_proba_concentration} Under \ref{ass:entries_cdts}, \ref{ass:C_N}, \ref{ass:eigenvalue_conv_Gamma}, \ref{ass:C_diagonal}, \ref{ass:Gamma_diagonal} and \ref{ass:bound_cdt}, for any $j\ge 1$ and any $\epsilon>0$, with overwhelming probability,
$$\left|\lambda_j(S_N(\Gamma))-a_j\int x\dd\nu_C(x)\right|<\epsilon.$$
\end{lemma}
For the definition of "overwhelming probability" or "tiny probability", refer to Definition~\ref{def:proba_termino}.
\begin{proof} We can repeat the first part of the proof of Proposition~\ref{prop:conv_eigen_SN} and we just need to verify that each "high probability" can be replaced by "overwhelming probability" under the assumptions of this lemma. Let $\tilde\Gamma_n, A_{\ell}$ be defined as above. From Theorem~3.1 of \cite{yin1986limiting}, we know that with overwhelming probability,
$$\|\frac{1}{N}C_N^{1/2}Z(\Gamma_n-\tilde \Gamma_n)Z^*C_N^{1/2}\|\le \epsilon\,.$$
Then we only need to prove that under the assumptions of this lemma, for any $\epsilon>0$, with overwhelming probability, 
$$\left| \frac{1}{N}u_i^*Z^*C_NZu_k-\E \frac{1}{N}u_i^*Z^*C_NZu_k\right|\le\epsilon\,.$$
As $C_N,\Gamma_n$ are diagonal, the above inequality is actually
\begin{equation}
    \frac{1}{N}\left|\sum_{p=1}^N c_p(\overline{Z_{p,i}}Z_{p,k}-\E\overline{Z_{p,i}}Z_{p,k}) \right|\le\epsilon\,. \label{eq:lemma_proof_ineq}
\end{equation}
Let 
$$\sigma^2=\sum_{p=1}^N\var (c_p\overline{Z_{p,i}}Z_{p,k})=\var (\overline{Z_{1,i}}Z_{1,k})\tr C_N^2.$$
Note that $\var (\overline{Z_{1,i}}Z_{1,k})=1$ if $i\ne k$, and $\var (\overline{Z_{1,i}}Z_{1,k})=\E|Z_{1,1}|^4-1$ if $i=k$. We assume that $\var (\overline{Z_{1,i}}Z_{1,k})\ne 0$, because otherwise we have $i=k$ and $|Z_{p,i}|^2=1$ almost surely, then \eqref{eq:lemma_proof_ineq} holds almost surely, and there is nothing to prove. 

Using Bennett's inequality (8b) of \cite{bennett1962probability}, for any $t>0$, one has
$$\P\left(\left|\sum_{p=1}^N c_p(\overline{Z_{p,i}}Z_{p,k}-\E\overline{Z_{p,i}}Z_{p,k}) \right|>t\sigma\right)\le 2e^{t(\sigma/(n\varepsilon_n^2))}\left(1+t\frac{n\varepsilon_n^2}{\sigma}\right)^{-t(\sigma/(n\varepsilon_n^2)+(\sigma/(n\varepsilon_n^2))^2)}.$$
As $\sigma^2/N\to \var (\overline{Z_{1,i}}Z_{1,k})\int x^2\dd\nu_C(x)\ne 0$, let $t$ be such that $t\sigma=N\epsilon$, then we have
$$\P\left(\frac{1}{N}\left|\sum_{p=1}^N c_p(\overline{Z_{p,i}}Z_{p,k}-\E\overline{Z_{p,i}}Z_{p,k})\right|>\epsilon\right)\le e^{a/\varepsilon_n^2}(1+bn\varepsilon_n^2)^{-c/\varepsilon_n^2-d/(\sqrt{n}\varepsilon_n^4)}$$
where $a,b,c,d$ are some positive constants. Because $\varepsilon_n$ is an almost sure upper bound of $|Z_{i,k}|/\sqrt{n}$, we can assume that $\sqrt{n}\varepsilon_n^2\to\infty$. Then for any fixed $M>0$, and for large enough $N,n$,
$$a/\varepsilon_n^2-(c/\varepsilon_n^2+d/(\sqrt{n}\varepsilon_n^4))\ln{(1+bn\varepsilon_n^2)}<-M\ln n.$$
Then the result follows.
\end{proof}

\section{Proof of Theorem~\ref{th:diag_clt}} \label{sec:proof_diag_clt}
In this section and Section~\ref{sec:proof_generalclt}, in order to simplify the notation, we omit the subscription $N$ and $n$ of $C_N$ and $\Gamma_n$, so they are just denoted as $C$ and $\Gamma$. We also simplify the notation $S_N(\Gamma)$ as $S$, and denote the eigenvalues of $S$ by
$$\tlambda_1\ge \tlambda_2 \ge \dots \ge \tlambda_N\,.$$

We prove the CLT for largest eigenvalues of $S$ in the following steps. First we truncate, recenter and rescale the entries of $Z$ so that $|Z_{i,j}|\le \varepsilon_n \sqrt{n}$ where $\varepsilon_n$ is a sequence of positive numbers tending to $0$. The truncation step is identical to the approach used in the proof of Theorem~1.1 of \cite{bai2004clt}, from where we know that this does not affect the result. So from now on we assume that \ref{ass:bound_cdt} holds.

Then in order to prove the weak convergence of $\Lambda_m(\Gamma)$, it suffices to prove that for any fixed vector $(b_1,\dots,b_m)^\tran\in\R^m$, we have
\begin{equation}
\P\left(\sqrt{N}\left(\frac{\tlambda_j}{t_j}-\theta_j\right)<b_j \quad \text{for}\quad 1\le j\le m \right)\xrightarrow[N,n\to\infty]{} \P\left(\mathbf G_j<b_j\quad \text{for}\quad 1\le j\le m\right)\,, \label{eq:cv_proba}
\end{equation}
where the random vector $(\mathbf G_1,\dots,\mathbf G_m)^\tran$ follows the limiting distribution $\gNr(0,(\E|Z_{1,1}|^4-1)\int x^2\dd\nu_C(x) I_m)$. For each $1\le j\le m$, we prove that the inequality 
\begin{equation}
    \sqrt{N}\left(\frac{\tlambda_j}{t_j}-\theta_j\right)<b_j \label{eq:to_rewrite}
\end{equation}
is equivalent to
$$\Yc_j+o_P(1)<b_j$$
for some random variable $\Yc_j$ which is expressed by the entries of $Z$. Then we determine the limiting distribution of $(\Yc_1,\dots,\Yc_m)^\tran$. Then the result follows by using Slutsky's Theorem.

\paragraph{Reformulation of the eigenvalue inequality \eqref{eq:to_rewrite}.} We begin to rewrite the inequality \eqref{eq:to_rewrite}. For further use we temporarily do not suppose that $C$ and $\Gamma$ are diagonal. So this part is shared with Section~\ref{sec:proof_generalclt}. We suppose that $\Gamma=UDU^*$ where $D=\diag(t_1,\dots,t_n)$ and $U=(u_1,\dots,u_n)$ is unitary. By normalizing $C, \Gamma$ and $S$, we assume without loss of generality that $\int x\dd\nu_C(x)=1$ and $a_j=t_j=1$. We set
$$D_{(j)}=\diag(t_1,\dots,t_{j-1},0,t_{j+1},\dots,t_n), \quad \Gamma_{(j)}=UD_{(j)}U^*, \quad\text{and}\quad S_{(j)}=\frac{1}{N}C^{1/2}Z\Gamma_{(j)}Z^*C^{1/2}.$$
Then $\tlambda_j$ satisfies the equation
\begin{equation}
    \det(\tlambda_j I-S_{(j)}-N^{-1}C^{1/2}Zu_ju_j^*Z^*C^{1/2})=0. \label{eq:det_equation}
\end{equation}
Under \ref{ass:spectral_gap}, applying Proposition~\ref{prop:conv_eigen_SN} to both $S$ and $S_{(j)}$, for a small enough $\epsilon>0$, with high probability, for $i=1,\dots,j-1$ we have
$$\tlambda_i, \lambda_i(S_{(j)})\in \left[a_i-\epsilon, a_i+\epsilon\right],$$
and since the $j$th largest eigenvalue of $\Gamma_{(j)}$ is $\lambda_{j+1}(\Gamma_n)$, which tends to $a_{j+1}$, we have
$$\tlambda_j\in \left[1-\epsilon, 1+\epsilon\right], \quad \tlambda_{j+1},\lambda_j(S_{(j)})\in \left[a_{j+1}-\epsilon, a_{j+1}+\epsilon\right].$$

We denote the above evenement by $\Omega$. Suppose that $\Omega$ happens. Then $\tlambda_j$ is not the eigenvalue of $S_{(j)}$, and the equation \eqref{eq:det_equation} is equivalent to
\begin{equation}
    \det(I-N^{-1}C^{1/2}Zu_ju_j^*Z^*C^{1/2}(\tlambda_j I-S_{(j)})^{-1})=0.
\end{equation}
Note that the matrix $N^{-1}C^{1/2}Zu_ju_j^*Z^*C^{1/2}(\tlambda_j I-S_{(j)})^{-1}$ is of rank one, so the equation is in fact
$$\Upsilon(\tlambda_j):=1-\frac{1}{N}u_j^*Z^*C^{1/2}(\tlambda_j I-S_{(j)})^{-1}C^{1/2}Zu_j=0.$$

Moreover, note that $\Upsilon(\lambda)>0$ for $\lambda$ large enough. Note also that
$$\Upsilon(\lambda)=\frac{\det(\lambda I-S)}{\det(\lambda I-S_{(j)})}$$
for $\lambda\in\R$, and with $\Omega$ holds, the denominator and the numerator of $\Upsilon(\lambda)$ change sign $(j-1)$ times respectively on $[1+\epsilon,\infty)$. So we deduce that 
$$\Upsilon\left(1+\epsilon\right)>0,\quad \Upsilon\left(1-\epsilon\right)<0,$$
and $\Upsilon(\lambda)$ changes sign in $[1-\epsilon,1+\epsilon]$ exactly at $\tlambda_j$. Thus for large enough $N,n$ such that 
$$\eta_j:=\theta_j+b_j/\sqrt{N}\in [1-\epsilon,1+\epsilon],$$
the inequality
$$\tlambda_j<\eta_j$$
is equivalent to
\begin{equation}
    \frac{1}{N}u_j^*Z^*C^{1/2}(\eta_jI-S_{(j)})^{-1}C^{1/2}Zu_j<1\,. \label{eq:ineq_equiv_1}
\end{equation}

If it happens that $\Gamma_{(j)}=0$, then $S_{(j)}=0$ and $\theta_j=1$, and the inequality \eqref{eq:ineq_equiv_1} is in fact
\begin{equation}
    \frac{1}{\sqrt{N}}u_j^*Z^*CZu_j-\sqrt{N}<b_j\,. \label{eq:ineq_equiv_1prime}
\end{equation}
Let the LHS of the above inequality be $\Yc_j$, then the procedure of rewriting \eqref{eq:to_rewrite} is complete. 

We now assume that $\Gamma_{(j)}\ne 0$, then we have $\mu^{\Gamma_{(j)}}\ne \delta_0$. We recall some results from \cite{couillet2014analysis}. By Proposition~1.1 of \cite{couillet2014analysis}, for any $z\in\C_+$, the system of equations
\begin{equation}
    \begin{cases}
    g^0_{\Gamma_{(j)}}=\frac{n}{N}\int\frac{t}{z(1-g^0_C t)}\dd\mu^{\Gamma_{(j)}}(t) \\
    g^0_{C}=\int\frac{t}{z(1-g^0_{\Gamma_{(j)}}t)}\dd\mu^{C}(t)
    \end{cases} \label{eq:sys_equation_g}
\end{equation}
has a unique solution $(g^0_{\Gamma_{(j)}}(z),g^0_{C}(z))$ such that $\Im g^0_{\Gamma_{(j)}}(z)<0, \Im g^0_{C}(z)<0$. Define
\begin{equation}
    F(g,z):=\frac{1}{N}\sum_{i=1}^N \frac{c_i}{z-(\frac{1}{N}\sum_{k\ne j}\frac{t_k}{1-gt_k})c_i}-g\,. \label{eq:def_F}
\end{equation}
By plugging the first equation of \eqref{eq:sys_equation_g} into the second one, and replacing the discrete integrals by sums, we can see that $g^0_{C}(z)$ is the unique solution of the equation
$$F(g^0_{C}(z),z)=0,$$
such that $\Im g_{C}(z)<0$. Let
$$s^0_C(z):=\int\frac{1}{z(1-g^0_{\Gamma_{(j)}}(z)t)}\dd\mu^{C}(t), \quad s_{\Gamma_{(j)}}^0(z):=\int\frac{1}{z(1-g^0_C(z)t)}\dd\mu^{\Gamma}(t).$$
Then $s_C^0$ (resp. $s_\Gamma^0$) is the Cauchy-Stieltjes transform of the probability measure $\mu^0_{C}$ (resp. $\mu^0_{\Gamma_{(j)}}$), which is the asymptotic equivalent of $\mu^{S_{(j)}}$ (resp. $\mu^{\frac{1}{N}\Gamma_{(j)}^{1/2}Z^*CZ\Gamma_{(j)}^{1/2}}$). See also (1.6)-(1.10) of \cite{zhidong2016central}. Moreover by Lemma~3.3 of \cite{couillet2014analysis}, for any $x\in\R\backslash\{0\}$, the limit $\lim_{z\in\C_+\to x}g^0_C(z)$ exists. Let $\gamma,\tilde\gamma$ be defined as
$$\gamma(z,g)=\frac{1}{N}\sum_{k\ne j}\frac{t_k^2}{z^2(1-gt_k)^2};$$
$$\tilde\gamma(z,g)=\frac{1}{N}\sum_{i=1}^N\frac{c_i^2}{z^2(1-\frac{c_i}{N}\sum_{k\ne j} \frac{t_k}{1-g t_k})^2}.$$
Then a main result of \cite{couillet2014analysis} says that a non-zero real number $x$ is outside the support of $\mu^0_C, \mu^0_{\Gamma_{(j)}}$, if and only if 
$$1-x^2\gamma(x,g^0_C(x))\tilde\gamma(x,g^0_{\Gamma_{(j)}}(x))>0.$$

Now we relate the function $g^0_C$ with $\theta_j$. By the definition of $\theta_j$, we have 
$$F(1,\theta_j)=0.$$
Because of the assumptions \ref{ass:C_N}, \ref{ass:eigenvalue_conv_Gamma}, also because the distance between $1$ and $\spec(\Gamma_{(j)})$ is bounded away from $0$, there is a complex neighborhood $B(1,\epsilon)$ of $1$ such that
$$\frac{c_i}{N}\sum_{k\ne j}\frac{t_k}{1-gt_k}\to 0$$
uniformly for $g\in B(1,\epsilon)$ and $i=1,\dots,N$. So there is a neighborhood $\mathcal U=B((1,1),\epsilon)\subset \C^2$ of $(1,1)$ ($\epsilon$ may take different values from one place to another) such that for $N,n$ large enough, the function $F$ is holomorphic in $\mathcal U$. Some calculation shows that
$$\frac{\partial F(g,z)}{\partial g}=z^2\gamma(z,g)\tilde\gamma(z,g)-1,$$
which is also holomorphic in $\mathcal U$ for $N,n$ large enough. Moreover  we have
$$\gamma(z,g)\to 0,\quad \tilde\gamma(z,g)\to \frac{1}{z^2}\int x^2\dd\nu_C(x)$$
uniformly for $(g,z)\in\mathcal U$ as $N,n\to\infty$. So for $N,n$ large enough and for $(g,z)\in\mathcal U$, we have
\begin{equation}
    \frac{\partial F(g,z)}{\partial g}<0 \label{eq:ineq_partialF}
\end{equation}
From Remark~\ref{rq:on_equation_G}, we have $\theta_j\to 1$. Whenever $(1,\theta_j)\in\mathcal U$, by holomorphic implicit function theorem \cite[Ch. 1, Th 7.6]{fritzsche2012holomorphic}, there exists a holomorphic function $g:z\mapsto g(z)$, defined in a complex neighborhood $B(\theta_j,\epsilon)$ of $\theta_j$, such that
$$F(g(z),z)=0.$$
Some calculations similar to those between (6)-(8) of \cite{couillet2014analysis} gives
$$\Im g(z)<0$$
for $z\in\C_+$. Then by the unicity of solution of the function $F$ for $z\in\C_+$, we have $g(z)=g^0_C(z)$ for $z\in\C_+\cup B(\theta_j,\epsilon)$. This proves that $g^0_C(\theta_j)=1$. 

Moreover from \eqref{eq:ineq_partialF} and Proposition~3.2 of \cite{couillet2014analysis} we deduce that for $N,n$ large enough, the point $\theta_j$ is in an interval $[1-\epsilon,1+\epsilon]$ who lies outside the support of $\mu^0_C, \mu^0_{\Gamma_{(j)}}$ for $N,n$ large enough. For any $z\in\C\backslash \supp \mu^0_{C}$, we have
\begin{equation}
    \frac{\dd g^0_C(z)}{\dd z}=-\frac{\frac{\partial F}{\partial z}}{\frac{\partial F}{\partial g}}=\frac{1}{z^2\gamma(z,g^0_C(z))\tilde\gamma(z,g^0_C(z))-1}\frac{1}{N}\sum_{i=1}^N \frac{c_i}{\left(z-\frac{c_i}{N}\sum_{k\ne j}\frac{t_k}{1-g^0_C(z)t_k}\right)^2}. \label{eq:derive_g}
\end{equation}

Note that $g^0_C(\theta_j)=1$, we rewrite the inequality \eqref{eq:ineq_equiv_1} as
\begin{equation}
    \frac{1}{N}u_j^*Z^*C^{1/2}(\eta_j I-S_{(j)})^{-1}C^{1/2}Zu_j-g^0_C(\eta_j)<g^0_C(\theta_j)-g^0_C(\eta_j)\,. \label{eq:ineq_equiv_2}
\end{equation}
By Lemma~3.4 of \cite{couillet2014analysis}, $g^0_C(x)\in\R$ for any $x\in\R\backslash\supp\mu^0_C$. Then we can apply Lagrange's Mean Value Theorem and get
\begin{equation}
    g^0_C(\theta_j)-g^0_C(\eta_j)=-\left.\frac{\dd g^0_C(z)}{\dd z}\right|_{z=\xi}\frac{b_j}{\sqrt{N}}, \label{eq:taylor_g}
\end{equation}
where $\xi$ is a number between $\theta_j$ and $\eta_j$. As $\theta_j$ and $\eta_j$ both tend to $1$, we also have $\xi\to1$. Then from the formula \eqref{eq:derive_g}, as $N,n\to\infty$, 
\begin{equation}
    \left.\frac{\dd g^0_C(z)}{\dd z}\right|_{z=\xi}\to -1. \label{eq:limit_dg}
\end{equation}
Plugging \eqref{eq:taylor_g}, \eqref{eq:limit_dg} into \eqref{eq:ineq_equiv_2}, and multiplying the two sides by $\sqrt{N}\eta_j$, also note that $\eta_j\to 1$, the inequality \eqref{eq:ineq_equiv_2} can be written as
\begin{equation}
    \frac{1}{\sqrt{N}}u_j^*Z^*C^{1/2}( I-\eta_j^{-1}S_{(j)})^{-1}C^{1/2}Zu_j-\sqrt{N}\eta_j g^0_C(\eta_j)<b_j+o(1)\,. \label{eq:ineq_equiv_3}
\end{equation}
Using the formula $A^{-1}-B^{-1}=A^{-1}(B-A)B^{-1}$, and letting
\begin{equation}
    Q:=C^{1/2}(\eta_j I-S_{(j)})^{-1}S_{(j)}C^{1/2}, \label{eq:def_Q_proof_dCLT}
\end{equation}
we can rewrite \eqref{eq:ineq_equiv_3} as
\begin{equation}
    \frac{1}{\sqrt{N}}(u_j^*Z^*CZu_j-\tr C)+\frac{1}{\sqrt{N}}u_j^*Z^*QZu_j-\sqrt{N}\eta_j g^0_C(\eta_j)+\frac{\tr C}{\sqrt{N}}<b_j+o(1)\,. \label{eq:ineq_equiv_4}
\end{equation}
Let 
$$\Yc_j=\frac{1}{\sqrt{N}}(u_j^*Z^*CZu_j-\tr C)$$
and 
$$R_n=\frac{1}{\sqrt{N}}u_j^*Z^*QZu_j-\sqrt{N}\eta_j g^0_C(\eta_j)+\frac{\tr C}{\sqrt{N}}.$$
In the following, we prove the CLT for $(\Yc_1,\dots,\Yc_m)^\tran$, and prove that $R_n=o_P(1)$, under the diagonal condition \ref{ass:C_diagonal}, \ref{ass:Gamma_diagonal} in this section, and under \ref{ass:contr_trace_1} or \ref{ass:contr_trace_2} in Section~\ref{sec:proof_generalclt}, respectively.

\paragraph{CLT for $(\Yc_1,\dots, \Yc_m)^\tran$ and estimation of the remainder $R_n$.}We now assume that $C, \Gamma$ are diagonal. Then we have
$$\Yc_j=\frac{1}{\sqrt{N}}\sum_{i=1}^N c_i(|Z_{i,j}|^2-1), \quad \text{for }\, j=1,\dots,m, $$
and by the CLT for independent random variables, we have
$$(\Yc_1,\dots, \Yc_m)^\tran \cvweak[N,n\to\infty]\gNr\left(0,(\E|Z_{1,1}|^4-1)\int x^2\dd\nu_C(x)I_m\right).$$

Now we prove that $R_n=o_P(1)$. Let $\bz_j$ denote the $j$th column of $Z$. As $C, \Gamma$ are diagonal, we have
$$\begin{aligned}R_n & = \frac{1}{\sqrt{N}}\bz_j^*Q\bz_j-\frac{1}{\sqrt{N}}\tr Q+\frac{1}{\sqrt{N}}\tr Q-\sqrt{N}\eta_j g^0_C(\eta_j)+\frac{\tr C}{\sqrt{N}} \\
&= \left(\frac{1}{\sqrt{N}}\bz_j^*Q\bz_j-\frac{1}{\sqrt{N}}\tr Q\right)+\left(\frac{1}{\sqrt{N}}\tr (I-\eta_j^{-1}S_{(j)})^{-1}C -\sqrt{N}\eta_j g^0_C(\eta_j)\right) \\
&= P_1+P_2,\end{aligned}$$
where $Q$ is defined in \eqref{eq:def_Q_proof_dCLT}. Note that $\bz_j$ and $Q$ are independent, then by Lemma~B.1 in \cite{zhidong2016central}, we have
$$\E_Q|P_1|^2\le \frac{\tr QQ^*}{N} K\E|Z_{1,1}|^4.$$
By Proposition~\ref{prop:asympt_eigenvect}, for any $\varepsilon>0$, with high probability, there is only a finite number of eigenvalues of $S_{(j)}$ larger than $\varepsilon$, and the distance between $\eta_j$ and the spectrum of $S_{(j)}$ is bounded away from $0$. So one has
$$\frac{\tr QQ^*}{N}=\frac{\|Q\|_F^2}{N}\le \|C\|^2\|(\eta_j I-S_{(j)})^{-1}\|^2\frac{\|S_{(j)}\|_F^2}{N}=o_P(1),$$
where $\|\cdot\|_F$ denotes the Frobenius norm, and recall that $\|AB\|_F\le \|A\| \|B\|_F$ for any matrices $A,B$. This proves that $\E_Q |P_1|^2=o_P(1)$. We also conclude that $P_1=o_P(1)$, because for any $\varepsilon>0$, 
$$\E_Q\ind_{|P_1|>\varepsilon}\le \varepsilon^{-2}\E_Q |P_1|^2\ind_{|P_1|>\varepsilon}=o_P(1),$$
and by Dominated Convergence Theorem, 
$$\P(|P_1|>\varepsilon)=\E\,\E_Q\ind_{|P_1|>\varepsilon}=o(1).$$

Then we prove that $P_2=o_P(1)$. We set 
$$g_C(z)=\frac{1}{N}\tr (zI-S_{(j)})^{-1}C.$$
Let $0<\epsilon<\min(1-a_{j+1}, a_{j-1}-1)/3$. We note that for $N,n$ large enough, $g^0_C(z)$ is analytic in the complex disc $B(\eta_j,\epsilon)$ of $\eta_j$, and with high probability, $g_C(z)$ is also analytic in this disc. If we can prove that for $w:=\eta_j+\im /n$,
\begin{equation}
    \tilde P_2:=\sqrt{N}g_C(w)-\sqrt{N}g^0_C(w)=o_P(1). \label{eq:P_2_oP1}
\end{equation}
Then, because
$$|P_2/\eta_j-\tilde P_2|\le \sqrt{N}|g_C(\eta_j)-g_C(w)|+\sqrt{N}|g^0_C(\eta_j)-g^0_C(w)|=O(1/\sqrt{n})$$
with high probability, and note that $\eta_j\to 1$, the result follows. The proof of \eqref{eq:P_2_oP1} uses some techniques of \cite[Section~3]{zhidong2016central}, and is postponed to Appendix~\ref{Append:P_2_oP1}.

\section{Proof of Theorem~\ref{th:joint_fluct}}
\label{sec:proof_generalclt}
In this section we extend our result to some non-diagonal $\Gamma$'s. We prove the CLT \eqref{eq:th_joint_fluct} under the condition \ref{ass:contr_trace_1} or \ref{ass:contr_trace_2}. 

We can prove that in any subsequence of the sequence
$$\left(\dlp(\Lambda_m(\Gamma),\gNr(0,\Sigma_m^{(N)}))\right)_{N,n\ge 1}$$
there is still a subsequence converging to $0$. Note that the entries of $\Sigma_m^{(N)}$ are bounded, we can assume that they converge, i.e.
$$\Sigma_m^{(N)}\to \Sigma_m.$$
Then we are led to proving that
$$\Lambda_m(\Gamma)\cvweak[N,n\to\infty]\gNr(0,\Sigma_m).$$

Let $U=(u_1,\dots,u_n)$ and $Y=ZU$. Let $\by_k=Zu_k$ be the $k$th column of $Y$. From the last section, the proof of the theorem can be done by proving that
$$(\Yc_1,\dots,\Yc_m)^\tran\cvweak[N,n\to\infty]\gNr(0,\Sigma_m)$$
where
$$\Yc_j:=\frac{1}{\sqrt{N}}(\by_j^*C\by_j-\tr C)=\frac{1}{\sqrt{N}}\sum_{i=1}^Nc_i(|Y_{i,j}|^2-1),$$
and that for any $j=1,\dots,m$,
$$R_n=\frac{1}{\sqrt{N}}\by_j^*Q\by_j-\sqrt{N}\eta_j g^0_C(\eta_j)+\frac{\tr C}{\sqrt{N}}=o_P(1),$$
where $Q$ is defined in \eqref{eq:def_Q_proof_dCLT}.

We use Cramér-Wold device to prove the CLT of the $m$-dimensional vector $(\Yc_1,\dots,\Yc_m)^\tran$. By a direct calculation, the covariance matrix of $(\Yc_1,\dots,\Yc_m)^\tran$ is exactly $\Sigma_m^{(N)}$ which tends to $\Sigma_m$ as we have assumed. Then for any fixed vector $\alpha:=(\alpha_1,\dots,\alpha_m)^\tran$, we prove that 
\begin{equation}
    \sum_{j=1}^m\alpha_j\Yc_j \cvweak \gNr(0,\alpha^\tran \Sigma_m \alpha). \label{eq:cramer_wold}
\end{equation}
If $\alpha^\tran \Sigma_m \alpha=0$, it means that $\var(\sum_{j=1}^m\alpha_j\Yc_j) =\alpha^\tran \Sigma_m^{(N)} \alpha\to 0$. Then $\sum_{j=1}^m\alpha_j\Yc_j =o_P(1)$ and hence \eqref{eq:cramer_wold} holds. Now we assume that $\alpha^\tran \Sigma_m \alpha\ne 0$ and prove that 
\begin{equation}
    \sum_{j=1}^m\alpha_j\Yc_j\cvweak[N,n\to\infty] \gNr(0,\alpha^\tran \Sigma_m \alpha)\,. \label{eq:clt_lindeberg}
\end{equation}
Note that the rows of $Y$ are i.i.d., then
$$\sum_{j=1}^m\alpha_j\Yc_j=\sum_{i=1}^N c_i\frac{\sum_{j=1}^m a_j(|Y_{i,j}|^2-1)}{\sqrt{N}}$$
is a weighted sum of $N$ i.i.d. random variables. We can use Lindeberg's CLT to prove \eqref{eq:clt_lindeberg}. To do so, we need to verify the Lindeberg condition
$$\frac{1}{N}\sum_{i=1}^N c_i^2\E\left|\sum_{j=1}^m \alpha_j(|Y_{i,j}|^2-1)\right|^2\ind_{\left|\sum_{j=1}^m \alpha_j(|Y_{i,j}|^2-1)\right|>\varepsilon\sqrt{N}}\to 0$$
as $N,n\to\infty$ for any $\varepsilon>0$. Since the quantities in the expectations are identically distributed for different $i$, and since 
$$\frac{1}{N}\sum_{i=1}^N c_i^2\to \int x^2\dd\nu_C(x)\in (0,\infty),$$
we only need to prove 
$$\E\left|\sum_{j=1}^m \alpha_j(|Y_{1,j}|^2-1)\right|^2\ind_{\left|\sum_{j=1}^m \alpha_j(|Y_{1,j}|^2-1)\right|>\varepsilon\sqrt{N}}\to 0.$$

Since $\E|Y_{1,j}|^2=1$ and since $\var |Y_{1,j}|^2$ is uniformly bounded, for any $\varepsilon>0$, the events
$$E_N:=\left\{\left|\sum_{j=1}^m \alpha_j(|Y_{1,j}|^2-1)\right|>\varepsilon\sqrt{N}\right\}$$
occur with low probability. By Minkowski's inequality, we have
$$\left(\E\left|\sum_{j=1}^m \alpha_j(|Y_{1,j}|^2-1)\right|^2\ind_{E_N}\right)^\frac{1}{2}\le \sum_{j=1}^m |\alpha_j|\left(\E\left| |Y_{1,j}|^2-1\right|^2\ind_{E_N}\right)^\frac{1}{2}\,.$$
Since $m$ is a fixed number, we only need to prove that for each $j=1,\dots,m$,
\begin{equation}
    \E\left| |Y_{1,j}|^2-1\right|^2\ind_{E_N}=\E |Y_{1,j}|^4\ind_{E_N}-2\E |Y_{1,j}|^2\ind_{E_N} +\P(E_N)\xrightarrow[N,n\to\infty]{} 0\,. \label{eq:proof:lindeberg_1}
\end{equation}
Since $\P(E_N)\to 0$ and from the uniform boundedness of $\E|Y_{1,j}|^4$ we have $\E |Y_{1,j}|^2\ind_{E_N}\to 0$, then \eqref{eq:proof:lindeberg_1} is equivalent to
$$\E |Y_{1,j}|^4\ind_{E_N}\xrightarrow[N,n\to\infty]{} 0\,.$$
This is a corollary of the following lemma.

\begin{lemma}\label{lem:lindeberg_cdt} Let $v=(v_1,\dots,v_n)\in \C^n$ such that  $\|v\|\le 1$. Let $(Z_k)_{k\ge 1}$ be a sequence of i.i.d. random variables satisfying $\E Z_k=0$, $\E|Z_k|^2=1$, $\E|Z_k|^4<\infty$. Let $E_n$ be a sequence of events such that $\P(E_n)\to 0$, then 
$$\lim_{n\to\infty}\E\left|\sum_{k=1}^n v_kZ_k\right|^4\ind_{E_n} = 0\,.$$
\end{lemma}

\begin{proof}We only prove the case where $v$ and $Z_k$ are real. For the complex case, it can be easily proved from the real case by separating real and complex parts, and then using Minkowski's inequality.

As all the random variables $|Z_k|^4$ are identically distributed and integrable, they are uniformly integrable. Thus we have
$$\max_{k}\E|Z_k|^4\ind_{E_n}\to 0.$$
Let $(e_n)$ be a sequence of positive numbers tending to $0$ such that 
$$\frac{\max_k\left(\E|Z_k|^4\ind_{E_n}\right)^\frac{1}{4}}{e_n^2}\to 0\quad\text{and}\quad n e_n^2\to \infty\quad\text{as}\quad n\to\infty\,.$$
Note that from $\|v\|\le 1$ we have
$$\#\{k\tq |v_k|>e_n\}\le \frac{1}{e_n^2}.$$
We write
$$\sum_{k=1}^N v_kZ_k=\sum_{k=1}^N v_k\ind_{|v_k|>e_n}Z_k+\sum_{k=1}^N v_k\ind_{|v_k|\le e_n}Z_k=:P_1+P_2\,.$$

By Minkowski's inequality, we have
$$\left(\E|P_1+P_2|^4\ind_{E_n}\right)^\frac{1}{4}\le \left(\E|P_1|^4\ind_{E_n}\right)^\frac{1}{4}+\left(\E|P_2|^4\ind_{E_n}\right)^\frac{1}{4}.$$
For the first part, using again Minkowski's inequality and noting that $|v_k|\le 1$, we get
$$\left(\E|P_1|^4\ind_{E_n}\right)^\frac{1}{4}\le \#\{i\,:\, |v_k|>e_n\}\max_k\left(\E|Z_k|^4\ind_{E_n}\right)^\frac{1}{4}\to 0.$$
For the second part, it suffices to prove that from any subsequence of $\left(\E|P_2|^4\ind_{E_n}\right)_n$ we can extract a subsequence tending to $0$. From any subsequence of $(\sum_{k=1}^N |v_k|^2\ind_{|v_k|<e_n})_n$, there exists a convergent subsequence. So we can assume that 
$$\sum_{k=1}^N |v_k|^2\ind_{|v_k|<e_n}\to \sigma^2\le 1.$$
Then if we can prove that
\begin{equation}
    \lim_{M\to\infty}\varlimsup_{N\to\infty}\E|P_2|^4\ind_{|P_2|>M}=0\,, \label{eq:strong_lindeberg}
\end{equation}
the proof of the lemma will be complete due to the inequality
$$\E|P_2|^4\ind_{E_n}\le \E|P_2|^4\ind_{|P_2|>M}+\E|P_2|^4\ind_{E_n\cap\{|P_2|\le M\}}\le \E|P_2|^4\ind_{|P_2|>M}+M^4\P(E_n).$$ 

To prove \eqref{eq:strong_lindeberg}, by the equality $x^4=\min(x^4, M^4)+ x^4\ind_{|x|>M} -M^4\ind_{|x|>M}$ for any $x\in\R$ and $M>0$, we have
$$\E|P_2|^4\ind_{|P_2|>M}=\E|P_2|^4-\min(\E|P_2|^4, M^4)+M^4\P(|P_2|>M)\,.$$
Let $N\to\infty$, since $\max_k\{|v_k|\ind_{|v_k|\le e_n}\}\le e_n\to 0$, by Lindeberg's Theorem we have $P_2\cvweak \gNr(0,\sigma^2)$. Let $\mathbf G\sim\gNr(0,\sigma^2)$, then we have
$$\E|P_2|^4\to 3\sigma^4=\E|\mathbf G|^4, \quad \min(\E|P_2|^4, M^4)\to \min(\E|\mathbf G|^4, M^4), \quad\text{and}\quad \P(|P_2|>M)\to \P(|\mathbf G|>M),$$
where the first convergence is from direct calculation, the second and the third are from the fact that $P_2\cvweak G$ and that the function $x\mapsto \min(x^4, M^4)$ is continuous and bounded. So we have
$$\lim_{N\to\infty}\E|P_2|^4\ind_{|P_2|>M}=\E|\mathbf G|^4\ind_{|\mathbf \mathbf G|>M}\,.$$
Finally we take $M\to\infty$ and see that \eqref{eq:strong_lindeberg} holds.
\end{proof}

Next we prove that
$$R_n=\frac{1}{\sqrt{N}}u_j^*Z^*QZu_j-\sqrt{N}\eta_j g^0_C(\eta_j)+\frac{\tr C}{\sqrt{N}}=o_P(1)$$
under the conditions \ref{ass:C_diagonal} and \ref{ass:contr_trace_1} or \ref{ass:contr_trace_2}, where $Q$ is defined in \eqref{eq:def_Q_proof_dCLT}.

We assume that \ref{ass:C_diagonal} and \ref{ass:contr_trace_1} hold. From the equation satisfied by $g^0_C$, we have
$$\begin{aligned}\sqrt{N}\eta_j g^0_C(\eta_j)-\frac{\tr C}{\sqrt{N}} &= \frac{1}{\sqrt{N}}\sum_{k\ne j}\frac{t_k}{1-g^0_C(\eta_j)t_k}\frac{1}{N}\sum_{i=1}^N\frac{c_i^2}{\eta_j-(\frac{1}{N}\sum_{k\ne j}^n\frac{t_k}{1-g^0_C(\eta_j)t_k})c_i} \\ 
 &= o(1).\end{aligned}$$
Now it suffices to prove that
$$\frac{1}{\sqrt{N}}u_j^*Z^*C^{1/2}\left[(\eta_j I-S_{(j)})^{-1}S_{(j)}\right]C^{1/2}Zu_j=o_P(1).$$
Note that with high probability, we have
$$\|(\eta_j I-S_{(j)})^{-1}\|\le K$$
for some $K>0$. Then the matrices
$$KS_{(j)}-(\eta_j I-S_{(j)})^{-1}S_{(j)} \quad\text{and}\quad KS_{(j)}+(\eta_j I-S_{(j)})^{-1}S_{(j)}$$
are both positive semi-definite. Then we have
$$\begin{aligned} & \left|\frac{1}{\sqrt{N}}u_j^*Z^*C^{1/2}\left[(\eta_j I-S_{(j)})^{-1}S_{(j)}\right]C^{1/2}Zu_j\right| \\
\le & \frac{K}{\sqrt{N}}u_j^*Z^*C^{1/2}S_{(j)}C^{1/2}Zu_j 
= \frac{K}{N\sqrt{N}}u_j^*Z^*CZ\Gamma_{(j)}Z^*CZu_j.\end{aligned}$$
To prove that this is $o_P(1)$, it suffices to prove that
\begin{equation}
    \frac{1}{N\sqrt{N}}\E u_j^*Z^*CZ\Gamma_{(j)}Z^*CZu_j=o(1). \label{eq:proof_expect1_o1}
\end{equation}
Denote the entries of $\Gamma_{(j)}$ by $\Gamma_{i,k}$. By a simple calculation, one has
\begin{equation}
\begin{aligned} &\frac{1}{N\sqrt{N}}\E u_j^*Z^*CZ\Gamma_{(j)}Z^*CZu_j \\
= & \frac{1}{N\sqrt{N}}\left(\sum_{i,k}|u_{j,k}|^2c_i^2\Gamma_{k,k}(\E|Z_{1,1}|^4-|\E Z_{1,1}^2|^2-2)+\sum_{i_1,i_2,k_1,k_2}c_{i_1}c_{i_2}\overline{u}_{j,k_1}\Gamma_{k_1,k_2}u_{j,k_2} \right. \\
 & \left. +\sum_{i,k_1,k_2}c_i^2\overline{u}_{j,k_1}\Gamma_{k_2,k_1}u_{j,k_2}|\E Z_{1,1}^2|^2+\sum_{i,k}c_i^2\Gamma_{k,k}\right).\end{aligned} \label{eq:term1_expect}    
\end{equation}
Note that
\begin{equation}
    \sum_{k_1,k_2}\overline{u}_{j,k_1}\Gamma_{k_1,k_2}u_{j,k_2}=u_j^*\Gamma_{(j)}u_j=0, \label{eq:orthogonal_uGamma}
\end{equation}
so the second term of the above sum is zero. Also, if we use $\overline{u}_j$ to denote the conjugate of the vector $u_j$, we have
$$\left|\sum_{k_1,k_2}\overline{u}_{j,k_1}\Gamma_{k_2,k_1}u_{j,k_2}\right|=|u_j^\tran \Gamma_{(j)}\overline{u}_j|=\left|\sum_{k\ne j} t_k(u_j^\tran u_ku_k^*\overline{u}_j)\right|\le \tr \Gamma_{(j)}.$$
Therefore by \ref{ass:C_N} and \ref{ass:contr_trace_1}, the limit \eqref{eq:proof_expect1_o1} is proved.

We now assume that \ref{ass:contr_trace_2} and \ref{ass:C_diagonal} hold. Using the formula $A^{-1}-B^{-1}=A^{-1}(B-A)B^{-1}$, and the inequality
$$\frac{1}{N^{3/2}}\left(\sum_{k=1}^n t_k\right)^2\le \frac{n}{N^{3/2}}\sum_{k=1}^n t_k^2,$$
we can write
$$\begin{aligned} R_n = & \left(\frac{1}{N^{3/2}\eta_j}u_j^*Z^*CZ\Gamma_{(j)}Z^*CZu_j- \frac{1}{N^{3/2}\eta_j}\sum_{i,k} t_k c_i^2\right) \\
 & + \frac{1}{\sqrt{N}\eta_j}u^*_jZ^*C^{1/2}(\eta_jI-S_{(j)})^{-1}S_{(j)}^2C^{1/2}Zu_j+O(\frac{1}{\sqrt{N}}\sum_k t_k^2) \\
 =:& P'_1+P'_2+o(1).\end{aligned}$$
The calculation \eqref{eq:term1_expect} also shows that
$$\begin{aligned}|\E P'_1|\le & \frac{K}{\sqrt{N}}\sum_{k=1}^n|u_{j,k}|^2\Gamma_{k,k}+\frac{K}{\sqrt{N}}\sum_{k_1,k_2}|\overline{u}_{j,k_1}u_{j,k_2}\Gamma_{k_2,k_1}| \\
    \le & \frac{K}{\sqrt{N}}\left(\sum_{k=1}^n|u_{j,k}|^4 \sum_{k=1}^n\Gamma_{k,k}^2\right)^{1/2} + \frac{K}{\sqrt{N}}\left(\sum_{k_1,k_2}|\overline{u}_{j,k_1}|^2|u_{j,k_2}|^2\sum_{k_1,k_2}|\Gamma_{k_2,k_1}|^2\right)^{1/2} \\
    =& o(1).\end{aligned}$$
To prove $P'_1=o_P(1)$, it suffices to prove that 
\begin{equation}
    \frac{1}{N^3}\var\left(u_j^*Z^*CZ\Gamma_{(j)}Z^*CZu_j\right)=o(1). \label{eq:variation_o1}
\end{equation}
Recall that $Y=Z(u_1,\dots,u_n)$ and $\by_k=Zu_k$. Note that the rows of $Y$ are independent, and that $\by_1,\dots,\by_n$ are decorrelated. Let $D_{(j)}=\diag(t_1,\dots,t_{j-1},0,t_{j+1},\dots,t_n)$. Then by some calculation, we have
$$\begin{aligned}     
 &\frac{1}{N^3}\var\left(\by_j^*CYD_{(j)}Y^*C\by_j\right) \\
=& \frac{1}{N^3}\left(\E(\by_j^*CYD_{(j)}Y^*C\by_j)^2-(\E \by_j^*CYD_{(j)}Y^*C\by_j)^2\right)\\
=& \frac{1}{N^3}\left(\E\left(\sum_{k_1,k_2\ne j}t_{k_1}t_{k_2}\by_j^*C\by_{k_1} \by_{k_1}^*\by_j\by_j^*C\by_{k_2} \by_{k_2}^*\by_j\right)-\left(\E\sum_{k\ne j}t_{k}\by_j^*C\by_k \by_k^*C\by_j\right)^2\right) \\
\le & \frac{1}{N^2} \sum_{k_1,k_2\ne j}t_{k_1}t_{k_2}\E|Y_{1,j}|^4|Y_{1,k_1}|^2|Y_{1,k_2}|^2 \frac{\tr C^2}{N} + \frac{1}{N}\sum_{k_1,k_2\ne j}t_{k_1}t_{k_2}|\E \overline{Y_{1,j}}^2Y_{1,k_1}Y_{1,k_2}|^2(\frac{\tr C}{N})^2 \\
    & +\frac{1}{N}\sum_{k_1,k_2\ne j}t_{k_1}t_{k_2}|\E|Y_{1,j}|^2Y_{1,k_1}\overline{Y_{1,k_2}}|^2(\frac{\tr C}{N})^2 \,.
\end{aligned}$$
Using Lemma~2.1 in \cite{bai1998no} and Holder's inequality,  $\E|Y_{1,j}|^4|Y_{1,k_1}|^2|Y_{1,k_2}|^2$ is uniformly bounded. Also we have $\frac{1}{N^2} \sum_{k_1,k_2\ne j}t_{k_1}t_{k_2}=\left(\frac{1}{N}\sum_{k\ne j}t_k\right)^2\to 0$, so the first term of the above sum is negligible. We only need to prove that
\begin{equation}
    \frac{1}{N}\sum_{k_1,k_2\ne j}t_{k_1}t_{k_2}|\E Y_{1,j}^2\overline{Y_{1,k_1}Y_{1,k_2}}|^2\to 0, \quad\text{and}\quad \frac{1}{N}\sum_{k_1,k_2\ne j}t_{k_1}t_{k_2}|\E|Y_{1,j}|^2Y_{1,k_1}\overline{Y_{1,k_2}}|^2\to 0\,. \label{eq:two_small1}
\end{equation}
We now prove the first convergence. By simple algebra, we have
$$\E \overline{Y_{1,j}}^2Y_{1,k_1}Y_{1,k_2}=\sum_{i=1}^n \overline{u_{j,i}}^2u_{k_1,i}u_{k_2,i}(\E|Z_{1,1}|^4-|\E Z_{1,1}^2|^2-2)+\left(\sum_{j=1}^n\overline{u_{j,i}}^2\right)\left(\sum_{i=1}^n u_{k_1,i}u_{k_2,i}\right)|\E Z_{1,1}^2|^2\,.$$
By the elementary inequality $|a+b|^2\le 2|a|^2+2|b|^2$, we only need to prove
\begin{equation}
    \frac{1}{N}\sum_{k_1,k_2\ne j}t_{k_1}t_{k_2}\left|\sum_{i=1}^n \overline{u_{j,i}}^2u_{k_1,i}u_{k_2,i}\right|^2\to 0, \quad\text{and}\quad \frac{1}{N}\sum_{k_1,k_2\ne j}t_{k_1}t_{k_2}\left|\sum_{i=1}^n u_{k_1,i}u_{k_2,i}\right|^2\to 0\,. \label{eq:two_small2}
\end{equation}
To prove the first convergence in \eqref{eq:two_small2}, we have
\begin{equation}
    \begin{aligned}\frac{1}{N}\sum_{k_1,k_2\ne j}t_{k_1}t_{k_2}\left|\sum_{i=1}^n \overline{u_{j,i}}^2u_{k_1,i}u_{k_2,i}\right|^2 & \le \frac{t_1^2}{N}\sum_{k_1,k_2}\left|\sum_{i=1}^n \overline{u_{j,i}}^2u_{k_1,i}u_{k_2,i}\right|^2 \\
    &= \frac{t_1^2}{N}\sum_{k_1,k_2}\sum_{i_1,i_2} \overline{u_{j,i_1}}^2u_{j,i_2}^2u_{k_1,i_1}\overline{u_{k_1,i_2}}u_{k_2,i_1}\overline{u_{k_2,i_2}} \\
    &= \frac{t_1^2}{N}\sum_{i_1,i_2} \overline{u_{j,i_1}}^2u_{j,i_2}^2\left(\sum_{k=1}^N u_{k,i_1}\overline{u_{k,i_2}}\right)^2 \\
    &=\frac{t_1^2}{N}\sum_{i=1}^n |u_{j,i}|^4\to 0\,.\end{aligned}
    \label{eq:proof:var_1_1_1}
\end{equation}
To prove the second convergence in \eqref{eq:two_small2}, we take an arbitrary small $\varepsilon>0$. Then by the assumption \ref{ass:eigenvalue_conv_Gamma}, the number of $t_k$ larger than $\varepsilon$ is finite. So we have
$$\begin{aligned}\frac{1}{N}\sum_{k_1,k_2\ne j}t_{k_1}t_{k_2}\left|\sum_{i=1}^n u_{k_1,i}u_{k_2,i}\right|^2 &\le \frac{\varepsilon^2}{N}\sum_{k_1,k_2}\left|\sum_{i=1}^n u_{k_1,i}u_{k_2,i}\right|^2 +O(N^{-1}) \\
& = \frac{\varepsilon^2}{N}\sum_{i_1,i_2}\left(\sum_{k=1}^N u_{k,i_1}\overline{u_{k,i_2}}\right)^2 +O(N^{-1})\\
& =\varepsilon^2+O(N^{-1})\,. \end{aligned}$$
Thus we have proved \eqref{eq:two_small2}, implying the first convergence of \eqref{eq:two_small1}. 

We then prove the second part of \eqref{eq:two_small1}. We have 
\begin{eqnarray*}
\lefteqn{ \E |Y_{1,j}|^2 Y_{1,k_1}\overline{Y_{1,k_2}} } \\ 
&=& \sum_{i=1}^n |u_{j,i}|^2 \overline{u_{k_1,i}}u_{k_2,i}(\E|Z_{1,1}|^4-|\E Z_{1,1}^2|^2-2)+\delta_{k_1,k_2} \\
 & & +\left(\sum_{i=1}^n\overline{u_{j,i}u_{k_1,i}}\right)\left(\sum_{i=1}^n u_{j,i}u_{k_2,i}\right)|\E Z_{1,1}^2|^2\,,
\end{eqnarray*}
where $\delta_{k_1,k_2}$ is the Kronecker symbol. Using the inequality $|a+b+c|^2\le 3(|a|^2+|b|^2+|c|^2)$ we only need to prove 
\begin{equation}
    \frac{1}{N}\sum_{k_1,k_2\ne j}t_{k_1}t_{k_2}\left|\sum_{i=1}^n |u_{j,i}|^2 \overline{u_{k_1,i}}u_{k_2,i}\right|^2\to 0\,,\quad  \frac{1}{N}\sum_{k\ne j}t_k^2\to 0\,, \label{eq:proof:var_2_12}
\end{equation}
and
\begin{equation}
    \frac{1}{N}\sum_{k_1,k_2\ne j}t_{k_1}t_{k_2}\left|\left(\sum_{i=1}^n\overline{u_{j,i}u_{k_1,i}}\right)\left(\sum_{i=1}^n u_{j,i}u_{k_2,i}\right)\right|^2\to 0\,.\label{eq:proof:var_2_3}
\end{equation}
The first of \eqref{eq:proof:var_2_12} can be proved similarly as \eqref{eq:proof:var_1_1_1}, the second of \eqref{eq:proof:var_2_12} is a consequnce of \ref{ass:contr_trace_2}. To prove \eqref{eq:proof:var_2_3}, we have
\begin{eqnarray*}
\lefteqn{ \frac{1}{N}\sum_{k_1,k_2\ne j}t_{k_1}t_{k_2}\left|\left(\sum_{i=1}^n\overline{u_{j,i}u_{k_1,i}}\right)\left(\sum_{i=1}^n u_{j,i}u_{k_2,i}\right)\right|^2 } \\
&\le & \frac{t_1^2}{N}\sum_{k_1,k_2}\sum_{i_1,i_2,i_3,i_4}\overline{u_{j,i_1}u_{k_1,i_1}}u_{j,i_2}u_{k_1,i_2} u_{j,i_3}u_{k_2,i_3}\overline{u_{j,i_4}u_{k_2,i_4}} \\
& = & \frac{t_1^2}{N}\sum_{i_1,i_2,i_3,i_4}\overline{u_{j,i_1}}u_{j,i_2}u_{j,i_3}\overline{u_{j,i_4}}\left(\sum_{k=1}^N\overline{u_{k,i_1}}u_{k,i_2}\right) \left(\sum_{k=1}^N u_{k,i_3}\overline{u_{k,i_4}}\right) \\
& = & \frac{t_1^2}{N}\left(\sum_{i=1}^n |u_{j,i}|^2\right)^2 = \frac{t_1^2}{N}\to 0\,.
\end{eqnarray*}
Then $P'_1=o_P(1)$ is proved.

To prove that $P'_2=o_P(1)$, using the same argument leading to  \eqref{eq:proof_expect1_o1}, we only need to prove that
\begin{equation}
    \frac{1}{N^2\sqrt{N}}\E u_j^*Z^*CZ\Gamma_{(j)}Z^*CZ\Gamma_{(j)}Z^*CZu_j=o(1). \label{eq:proof_expect2_o1}
\end{equation}
Using the same notations as proving \eqref{eq:variation_o1}, by simple algebra, we have
\begin{equation}\begin{aligned}
\frac{1}{N^2\sqrt{N}} \E \by_j^*CYD_{(j)}Y^*CYD_{(j)}Y^*C\by_j \le & \frac{1}{N^{3/2}}\sum_{k}t_{k}^2\E|Y_{1,j}|^2|Y_{1,k}|^4 \frac{\tr C^3}{N} \\
  & + \frac{1}{N^{1/2}}\sum_{k}t_{k}^2\E|Y_{1,j}|^2|Y_{1,k}|^2 \frac{\tr C\tr C^2}{N^2} \\
    &+\frac{1}{N^{3/2}}\sum_{k_1,k_2}t_{k_1}t_{k_2}\E|Y_{1,i}|^2|Y_{1,k_1}|^2|Y_{1,k_2}|^2\frac{\tr C^3}{N} \\
    =& o(1).
\end{aligned}\label{eq:proof:expect_S2}\end{equation}
Then we have $P'_2=o_P(1)$. 

\begin{Rq}
As we have seen, the main difficulty of this proof is the convergence to $0$ in probability of a quadratic form $\by^*Q\by$ where $\by$ and $Q$ are not independent. In the expansion e.g. \eqref{eq:orthogonal_uGamma} or other expansions afterwards, we can see that the orthogonal relation between $u_j$ and $\Gamma_{(j)}$ is crucial for the result. Up to the date of submission of this article and to our best knowledge, we have not found any method other than moment expansions which can achieve the same or stronger result. The method used in \cite[Theorem~2.4]{cai2017limiting} could be a possible option. But instead of the orthogonality between $u_j$ and $\Gamma_{(j)}$, the proof in \cite{cai2017limiting} relies on the clear separation between spiked and non-spiked population eigenvalues and the quicker speed of the non-spiked eigenvalues converging to $0$, which is not satisfied by our model.

Potentially our method works also in the case of non-diagonal $C_N$, and more general $\Gamma_n$, especially for the Toeplitz matrices $T_n$ with parameter $\rho\in(-1,-3/4]$. But restricted to the complexity of proof and the length of the paper we may proceed in this direction in the forthcoming works.
\end{Rq}

\appendix
\section{Proof of \eqref{eq:P_2_oP1}}
\label{Append:P_2_oP1}
We write
$$\tilde P_2=\sqrt{N}(g_C(w)-\E g_C(w)) + \sqrt{N}(\E g_C(w)-g^0_C(w))=:P_{21}+P_{22}.$$
To prove that $P_{21}=o_P(1)$, we use the martingale decomposition. By reassigning $t_j=0$, we continue to denote the eigenvalues of $\Gamma_{(j)}$ by $t_1,\dots,t_n$. Let $\by_k$ be the $k$th column of $C^{1/2}Z$. Let $\E_k$ denote $\E_{\by_1,\dots,\by_k}$ and $\E_0=\E$. 
We denote
$$S=\frac{1}{N}\sum_{i=1}^n t_i\by_i\by_i^*,\quad S_{(k)}=\frac{1}{N}\sum_{i\ne k} t_i\by_i\by_i^*$$
$$D(w):=wI-S, \quad D_k(w):=wI-S_{(k)},$$
and
$$\zeta_k(w):=\by_k^*D_k(w)^{-1}\by_k-\tr D_k^{-1}(w)C, \quad \xi_k(w):=\by_k^*D_k^{-1}(w)CD_k^{-1}(w)\by_k - \tr(D_k^{-1}(w)C)^2,$$
$$\beta_k(w):=\frac{1}{1-N^{-1}t_k\by_k^*D_k^{-1}(w)\by_k}, \quad \tilde\beta_k(w):=\frac{1}{1-N^{-1}t_k\tr D_k^{-1}(w)C},$$
$$\phi_k(w):=\frac{1}{1-N^{-1}t_k\E\tr(D_k^{-1}(w)C)}, \quad \psi_k(w):=\frac{1}{1-N^{-1}t_k\E\tr(D^{-1}(w)C)}.$$

\begin{lemma}\label{lem:expect_bounds} Under the conditions of Theorem~\ref{th:diag_clt} with $Z_{i,j}$ satisfying \ref{ass:bound_cdt}, for any $p\ge 1$, there exists a constant $K_p$ such that for large enough $N,n$, we have
$$\E\|D^{-1}(w)\|^p\le K_p, \E\|D_k^{-1}(w)\|^p\le K_p, \E|\beta_k(w)|^p\le K_p, \E|\tilde\beta_k(w)|^p\le K_p;$$
$$|\phi_k(w)|\le K_p, \quad |\psi_k(w)|\le K_p.$$
\end{lemma}
\begin{proof}
The first two inequalities are due to the fact that with overwhelming probability the distance from $w$ to the spectra of $S$ and $S_{k}$ are bounded away from zero uniformly in $n,k$ (Lemma~\ref{lemma:high_proba_concentration}), so $\|D^{-1}(w)\|$ and $\|D_k^{-1}(w)\|$ are uniformly bounded with overwhelming probability; and with tiny probability, we use the general bound
$$\|D^{-1}(w)\|\le \|C\|\Im w = \|C\|n, \quad \|D_k^{-1}(w)\|\le  \|C\|\Im w = \|C\|n.$$
Therefore the expectations $\E\|D^{-1}(w)\|^p$, $\E\|D_k^{-1}(w)\|^p$ are uniformly bounded.

For the third one, the proof is identical to the proof of  Lemma~A.3 in \cite{zhidong2016central}, up to some adaptions. The first steps of the proof of Lemma~A.3 in \cite{zhidong2016central} give
$$\frac{1}{N}t_k\by_k^*D^{-1}(w)\by_k=\beta_k(w)-1.$$
Then we have
$$|\beta_k(w)|\le 1+\|D^{-1}(w)\| |\frac{1}{N}t_k\by_k^*\by_k|.$$
To prove that $\E|\frac{1}{N}t_k\by_k^*\by_k|^p$ is uniformly bounded for any fixed $p\ge 1$, with overwhelming probability, we can use the bound
$$|\frac{1}{N}t_k\by_k^*\by_k|=\left|\max_{\|f\|=1}f^*(S-S_{(k)})f\right|\le \left|\max_{\|f\|=1}f^*Sf\right|+\left|\max_{\|f\|=1}f^*S_{(k)} f\right|\le Ka_1;$$
and with tiny probability we use the general bound
$$|\frac{1}{N}t_k\by_k^*\by_k|\le Kn.$$
Then we have
$$\E|\beta_k(w)|^p \le 1+K\E\|D^{-1}(w)\|^p |\frac{1}{N}t_k\by_k^*\by_k|^p\le 1+K_p\E^{1/2}\|D^{-1}(w)\|^{2p} \E^{1/2}|\frac{1}{N}t_k\by_k^*\by_k|^{2p}\le K_p.$$

For the fourth one, we write
$$\tilde\beta_k(w)=\frac{1}{1-(Nw)^{-1}t_k\tr C-(Nw)^{-1}t_k\tr D_k^{-1}(w)S_{(k)}C}.$$
Note that for large enough $N,n$, $(Nw)^{-1}t_k\tr C$ has a positive distance to $1$. We now prove that with overwhelming probability, the term $(Nw)^{-1}t_k\tr D_k^{-1}(w)S_{(k)}C$ is small enough so that the denominator of $\tilde\beta_k(w)$ is bounded away from $0$.

Let $\lambda_1,\dots,\lambda_N$ be the eigenvalues of $S_{(k)}$, with eigenvectors $v_1,\dots,v_N$ where $v_i=(v_{1,i},\dots,v_{N,i})^\tran$. Recall that $C=\diag(c_1,\dots,c_N)$. Then
\begin{equation}
    \left|\frac{1}{N}\tr D_k^{-1}(w)S_{(k)}C\right|\le \frac{1}{N}\sum_{j=1}^N\left|\frac{\lambda_j}{w-\lambda_j}\right|\sum_{i=1}^N c_i|v_{i,j}|^2\le \frac{\|C\|}{N}\sum_{j=1}^N\left|\frac{\lambda_j}{w-\lambda_j}\right|. \label{eq:tr_DBC_o1}
\end{equation}
Then by Lemma~\ref{lemma:high_proba_concentration}, with overwhelming probability, the above quantity is smaller than any fixed $\varepsilon>0$. 

With tiny probability, the above estimation does not hold. For the estimation of expectation, we should find a new estimation of $\tilde\beta_k(w)$. Note that if $|N^{-1}t_k\tr D_k^{-1}(w)C|\le 1/2$, then $|\tilde\beta_k(w)|\le 2$; else if $|N^{-1}t_k\tr D_k^{-1}(w)C|>1/2$ we have
$$\frac{1}{2}|\tilde\beta_k(w)|\le |N^{-1}t_k\tr D_k^{-1}(w)C||\tilde\beta_k(w)| \le \frac{|\tr D_k^{-1}(w)C|}{|\Im \tr D_k^{-1}(w)C|}\le \frac{2\|C\|n^2\max_{i}|w-\lambda_i|^2}{N^{-1}\tr C}\le Kn^5$$
where we estimate $\max_i|\lambda_i|^2$ as in the proof of Lemma~A.1 in \cite{zhidong2016central}. Thus finally we have
$$\E|\beta_k(w)|^p\le K_p.$$

By the same arguments, one can also verify the boundedness of $\phi_k(w)$ and $\psi_k(w)$.
\end{proof}

Then
$$\begin{aligned}\sqrt{N}P_{21} = & \tr D^{-1}(w)C-\E\tr D^{-1}(w)C \\
        = & \sum_{k=1}^n(\E_k-\E_{k-1})\tr (D^{-1}(w)-D_k^{-1}(w))C \\
    =& \frac{1}{N}\sum_{k=1}^n(\E_k-\E_{k-1})t_k\beta_k(w)\by_k^*D_k^{-1}(w)CD_k^{-1}(w)\by_k \\
    =& \frac{1}{N}\sum_{k=1}^n(\E_k-\E_{k-1})t_k\beta_k(w)\xi_k(w)  +\frac{1}{N}\sum_{k=1}^n(\E_k-\E_{k-1})t_k\beta_k(w)\tr(D_k^{-1}(w)C)^2 \\
    =:& J_1+J_2.\end{aligned}$$
Then we have
$$\begin{aligned}\E|J_1|^2 &= \frac{1}{N^2}\sum_{k=1}^n \E |(\E_k-\E_{k-1})t_k\beta_k(w)\xi_k(w)|^2 \\
    &\le \frac{K}{N^2}\sum_{k=1}^n t_k^2 \E |\beta_k(w)\xi_k(w)|^2 \\
    &\le \frac{K}{N^2}\sum_{k=1}^n t_k^2 \E^{1/2} |\beta_k(w)|^4 \E^{1/2}|\xi_k(w)|^4.\end{aligned}$$
By Lemma~\ref{lem:expect_bounds}, the expectation $\E|\beta_k(w)|^4$ is uniformly bounded; by Lemma~B.1 in \cite{zhidong2016central}, we have, for any $p\ge 1,q\ge 1$,
\begin{equation}
    \begin{aligned}\E^{1/q}|\xi_k(w)|^p & \le K_{p,q}\left[\left(\E|Z_{1,1}|^4\E\tr((D_k^{-1}(w)C)^2(CD_k^{-1}(\overline{w}))^2)\right)^{p/2}\right. \\ 
    & \quad +\left. \E|Z_{1,1}|^{2p}\E\tr((D_k^{-1}(w)C)^2(CD_k^{-1}(\overline{w}))^2)^{p/2}\right]^{1/q}\\
    & \le  K_{p,q}[N^{p/(2q)}+\varepsilon_n^{2p-6}n^{p-3}N]^{1/q}\le K_{p,q}(n^{\max\{p/2,(p-2)/q\}}).\end{aligned} \label{eq:quadratic_estim}
\end{equation}
So we have
$$\E|J_1|^2\le \frac{K}{N}\sum_{k=1}^n t_k^2=o(1).$$

For $J_2$, note that
$$\beta_k(w)=\tilde\beta_k(w)+\frac{1}{N}t_k\beta_k(w)\tilde\beta_k(w)\zeta_k(w),$$
then we have
$$J_2=\frac{1}{N^2}\sum_{k=1}^n(\E_k-\E_{k-1})t_k^2\beta_k(w)\tilde\beta_k(w)\zeta_k(w)\tr(D_k^{-1}(w)C)^2$$
Using the same arguments, we get 
$$\E|J_2|^2\le \frac{K}{N}\sum_{k=1}^n t_k^4=o(1).$$
Therefore in fact $P_{21}=o_P(1/\sqrt{N})$.

To prove that $P_{22}=o_P(1)$, we consider two cases:
\begin{enumerate}
    \item $\tr\Gamma/\sqrt{n}\to 0$; \label{enum:P_22_cond_0}
    \item $\tr\Gamma/\sqrt{n}\to a\ne 0$, where $a$ is a positive number or $\infty$. \label{enum:P_22_cond_n0}
\end{enumerate}
By extracting subsequences we can assume that one of these two cases holds. Indeed, we want to prove that for any $\varepsilon>0$, the following limit holds:
$$\P(|P_{22}|>\varepsilon)\to 0.$$
We can prove that from any subsequence of 
$$(\P(|P_{22}|>\varepsilon))_{n\ge 1},$$
there exists a subsequence converging to $0$. We know that from any subsequence of $(\Gamma_n)_n$, one can extract a subsequence $(\Gamma_{n_k})_{k\ge 1}$ satisfying one of the above two cases.

Suppose that Condition~\ref{enum:P_22_cond_0} holds. Then we have
\begin{multline}\sqrt{N}|\E g_C(w)-g^0_C(w)| = \frac{1}{\sqrt{N}}\left|\E\tr (wI-S)^{-1}C -\sum_{i=1}^N\frac{c_i}{w-\frac{c_i}{N}\sum_{k\ne j}\frac{t_k}{1-g^0_C(w)t_k}} \right| \\
= \frac{1}{\sqrt{N}|w|}\left|\E\tr (wI-S)^{-1}SC-\frac{1}{N}\sum_{k\ne j}\frac{t_k}{1-g^0_C(w)t_k}\sum_{i=1}^N\frac{c_i^2}{w-\frac{c_i}{N}\sum_{k\ne j}\frac{t_k}{1-g^0_C(w)t_k}}\right| \hfill\\
\le \frac{1}{\sqrt{N}|w|}\left|\E\tr (wI-S)^{-1}SC\right|+\frac{1}{N^{3/2}|w|}\left|\sum_{k\ne j}\frac{t_k}{1-g^0_C(w)t_k}\sum_{i=1}^N\frac{c_i^2}{w-\frac{c_i}{N}\sum_{k\ne j}\frac{t_k}{1-g^0_C(w)t_k}}\right| \hfill\\
= P_{221}+P_{222}\,. \hfill\end{multline}
One can see that 
$$\frac{1}{N}\left|\sum_{i=1}^N\frac{c_i^2}{w-\frac{c_i}{N}\sum_{k\ne j}\frac{t_k}{1-g^0_C(w)t_k}}\right|$$
is bounded, and 
$$\frac{1}{N^{1/2}|w|}\left|\sum_{k\ne j}\frac{t_k}{1-g^0_C(w)t_k}\right|=o(1),$$
thus $P_{222}=o(1)$. For $P_{221}$, by \eqref{eq:tr_DBC_o1} and the arguments afterwards, we have, for an arbitrary $p>1$,
$$\frac{1}{\sqrt{N}}|\E\tr D^{-1}(w)SC|\le \frac{K}{\sqrt{N}}\E\sum_{k=1}^N|\lambda_k|+o(n^{-p})=\frac{K}{\sqrt{N}}\E\tr S+O(n^{-1/2})=o(1),$$
where $\lambda_k$ are the eigenvalues of $S$, and we have used the equality
$$\E\tr S=\frac{1}{N}\tr C\tr \Gamma_{(j)}.$$

Now we suppose that Condition~\ref{enum:P_22_cond_n0} holds. We define
$$s(w):=\frac{1}{N}\tr D^{-1}(w), \quad s^{0}(w):=\frac{1}{N}\sum_{i=1}^N\frac{1}{w-\frac{c_i}{N}\sum_{k=1}^n\frac{t_k}{1-g^0_C(w)t_k}}.$$
Then 
\begin{equation}
    \begin{aligned}N(\E s_C(w)-s_C^{0}(w))=& \E\tr D^{-1}(w)-\sum_{i=1}^N\frac{1}{w-\frac{1}{N}\sum_{k=1}^nt_k\psi_k(w)c_i} \\
 &+\sum_{i=1}^N\frac{1}{w-\frac{c_i}{N}\sum_{k=1}^n\frac{t_k}{1-t_k\E g_C(w)}} -\sum_{i=1}^N\frac{1}{w-\frac{c_i}{N}\sum_{k=1}^n\frac{t_k}{1-g^0_C(w)t_k}} \\
 =& d_n+\sqrt{N}(\E g_C(w)-g^0_C(w))\frac{1}{\sqrt{N}}\sum_{k=1}^n\frac{t_k^2}{(1-g^0_C(w)t_k)(1-\E g_C(w)t_k)}\times \\
 &\frac{1}{N}\sum_{i=1}^N\frac{c_i}{(w-\frac{1}{N}\sum_{k=1}^nt_k\psi_k(w)c_i)(w-\frac{c_i}{N}\sum_{k=1}^n\frac{t_k}{1-g^0_C(w)t_k})}, \end{aligned} \label{eq:s_g_eq1}
\end{equation}
where $d_n:=\E g_C(w)-\sum_{i=1}^N\frac{1}{w-\frac{1}{N}\sum_{k=1}^nt_k\psi_k(w)c_i}$. On the other hand, following the calculation for (3.41) of \cite{zhidong2016central}, we have
 $$1-w\E s_C(w)=\frac{n}{N}-\frac{1}{N}\sum_{k=1}^n\E\beta_k(w).$$
By the definition of $s_C$ and the system of equations \eqref{eq:sys_equation_g}, we have
 $$1-w s_C^0(w)=\frac{n}{N}-\frac{1}{N}\sum_{k=1}^n\frac{1}{1-g^0_C(w)t_k}.$$
Taking the difference of the last two equalities, we obtain
\begin{multline}
    N(\E s_C(w)-s_C^0(w)) = w^{-1}\sum_{k=1}^n (\E\beta_k-\psi_k)+w^{-1}\sum_{k=1}^n\left(\psi_k-\frac{1}{1-g^0_C(w)t_k}\right) \\
    =w^{-1}\sum_{k=1}^n (\E\beta_k-\psi_k)+\sqrt{N}(\E g_C(w)-g^0_C(w))\frac{1}{\sqrt{N}}\sum_{k=1}^n\frac{t_k}{w(1-\E g_C(w)t_k)(1-g^0_C(w)t_k)}. \label{eq:s_g_eq2}
\end{multline}
Combining the equations \eqref{eq:s_g_eq1} and \eqref{eq:s_g_eq2}, we obtain
$$\begin{aligned} & \sqrt{N}(g_C(w)-g^0_C(w)) \times  \\
 & \frac{1}{\sqrt{N}}\left(\sum_{k=1}^n\frac{t_k}{w(1-\E g_C(w)t_k)(1-g^0_C(w)t_k)}-\sum_{k=1}^n\frac{t_k^2}{(1-g^0_C(w)t_k)(1-\E g_C(w)t_k)}\times \right.\\
  & \left.\frac{1}{N}\sum_{i=1}^N\frac{c_i}{(w-\frac{1}{N}\sum_{\ell=1}^nt_\ell\psi_\ell(w)c_i)(w-(\frac{1}{N}\sum_{\ell=1}^n\frac{t_\ell}{1-g^0_C(w)t_\ell})c_i)}\right) \\
 =& d_n-\frac{1}{w}\sum_{k=1}^n(\E \beta_k(w)-\psi_k(w)).\end{aligned}$$
Next we prove that $d_n=\E g_C(w)-\sum_{i=1}^N\frac{1}{w-\frac{1}{N}\sum_{k=1}^nt_k\psi_k(w)c_i}=o(1)$, $\sum_{k=1}^n(\E \beta_k(w)-\psi_k(w))=o(1)$, and that the multiplier 
\begin{equation}
    \begin{aligned}\frac{1}{\sqrt{N}}\left(\sum_{k=1}^n\frac{t_k}{w(1-\E g_C(w)t_k)(1-g^0_C(w)t_k)}-\sum_{k=1}^n\frac{t_k^2}{(1-g^0_C(w)t_k)(1-\E g_C(w)t_k)}\times \right.\\
 \left.\frac{1}{N}\sum_{i=1}^N\frac{c_i}{(w-\frac{1}{N}\sum_{\ell=1}^nt_\ell\psi_\ell(w)c_i)(w-(\frac{1}{N}\sum_{\ell=1}^n\frac{t_\ell}{1-g^0_C(w)t_\ell})c_i)}\right)\end{aligned} \label{eq:multiplier}
\end{equation}
is bounded away from $0$.

Repeating the calculations in Section~3.3 of \cite{zhidong2016central}, one can check that $d_n=o(1)$ and $\sum_{k=1}^n(\E \beta_k(w)-\psi_k(w))=o(1)$. The proof is similar so we omit the details. We just point out some adaptions due to the differences between the models. We define
$$W(w)=wI-\frac{1}{N}\sum_{k=1}^n t_k\psi_k(w)C.$$
Because $\|\frac{1}{N}\sum_{k=1}^n t_k\psi_k(w)C\|\le \frac{1}{N}\sum_{k=1}^n t_k|\psi_k(w)|\|C\|\to 0$, the matrix $W(w)$ is invertible for $N,n$ large enough, and $\|W^{-1}(w)\|$ is uniformly bounded. According to (3.33) of \cite{zhidong2016central} and the estimations afterwards, the result can be similarly deduced. We also remind that in Section~3.3 of \cite{zhidong2016central} the proof is made for Gaussian entries. Here we only assume that the entries $Z_{i,j}$ have finite sixth moment, so according to \eqref{eq:quadratic_estim}, we have for example the following estimation of $\mathcal H_2$ which is correspondingly defined in the equation next to (3.36) of \cite{zhidong2016central}:
$$\begin{aligned}\mathcal H_2&:=\frac{1}{N^3}t_k^3\phi_k^3(w)\E \beta_k(w)(\by_k^*D_k^{-1}(w)\by_k-\E\tr D_k^{-1}(w)C)^3 \\
    &\le \frac{K}{N^3}t_k^3\E^{1/2}|\beta_k(w)|^2\E^{1/2}|\by_k^*D_k^{-1}(w)\by_k-\E\tr D_k^{-1}(w)C|^6 \\
    &\le \frac{K}{N}t_k^3,\end{aligned}$$
thus by the formula next to (3.36) and the formula (3.38) of \cite{zhidong2016central}, we have
$$\left|\sum_{k=1}^n(\E \beta_k(w)-\psi_k(w))\right|\le \frac{K}{N}\sum_{k=1}^n(t_k^2 + t_k^3)=o(1).$$

To prove that the multiplier \eqref{eq:multiplier} is bounded from below, we recall that $w\to 1$, $g^0_C(w)\to 1$, $\E g_C(w)\to 1$, $\frac{1}{N}\sum_{\ell=1}^nt_\ell\psi_\ell(w)=o(1)$, and $\frac{1}{N}\sum_{\ell=1}^n\frac{t_\ell}{1-g^0_C(w)t_\ell}=o(1)$, thus for any $k=1,\dots,n$, as $N,n\to\infty$,
$$e_k:=\frac{1}{w(1-\E g_C(w)t_k)(1-g^0_C(w)t_k)}\to \frac{1}{(1-t_k)^2},$$
$$\begin{aligned}f_k & :=\frac{1}{(1-g^0_C(w)t_k)(1-\E g_C(w)t_k)}\times \frac{1}{N}\sum_{i=1}^N\frac{c_i}{(w-\frac{1}{N}\sum_{\ell=1}^nt_\ell\psi_\ell(w)c_i)(w-(\frac{1}{N}\sum_{\ell=1}^n\frac{t_\ell}{1-g^0_C(w)t_\ell})c_i)} \\ 
 & \to \frac{1}{(1-t_k)^2}.\end{aligned}$$
The above two limits are uniform in $k$, so we have, for any $\varepsilon>0$, for $N,n$ large enough,
$$\frac{1}{\sqrt{N}}\left|\sum_{k=1}^n (t_ke_k-t_k^2f_k)-\frac{t_k}{(1-t_k)}\right|\le \frac{\varepsilon}{\sqrt{N}}\sum_{k=1}^n t_k.$$
Therefore
$$\frac{1}{\sqrt{N}}\sum_{k=1}^n (t_ke_k-t_k^2f_k)\sim \frac{1}{\sqrt{N}}\sum_{k=1}^n \frac{t_k}{1-t_k}\sim \frac{1}{\sqrt{N}}\sum_{k=1}^n t_k,$$
which is lower bounded because we are just in the case where
$\frac{1}{\sqrt{n}}\sum_{k=1}^n t_k\to a\ne 0$. 

The proof of \eqref{eq:P_2_oP1} is complete.

\bibliography{main}{}
\bibliographystyle{plain}
\end{document}